\documentclass[a4paper,12pt]{article}
\usepackage[utf8]{inputenc}
\usepackage[english]{babel}
\usepackage[ruled,section]{algorithm}
\usepackage{makeidx}
\usepackage[font=small]{caption}
\usepackage{amsmath}
\usepackage{amsthm}
\usepackage{amsfonts}
\usepackage{amssymb}
\usepackage{mathrsfs}
\usepackage{graphicx}
\usepackage[margin=18mm]{geometry}
\usepackage{enumerate}
\usepackage{indentfirst}
\usepackage[none]{hyphenat}
\usepackage{color}
\usepackage{authblk}
\makeindex
\newcommand{\real}{\mathbb{R}}

\newcommand{\rN}{ {\mathbb{R}^N} }

\newcommand{\rmd}{\mathrm{d}}
\newcommand{\F}{\mathcal{F}}
\newcommand{\K}{\mathcal{K}}
\newcommand{\C}{\mathcal{C}}
\newcommand{\W}{\mathcal{W}}

\newcommand{\sH}{\mathcal{H}}
\newcommand{\tH}{\widetilde{\mathcal{H}}}

\newcommand{\M}{\mathcal{M}}

\numberwithin{equation}{section}
\newtheorem{theorem}{Theorem}[section]
\newtheorem{lem}[theorem]{Lemma}
\newtheorem{prop}[theorem]{Proposition}
\newtheorem{corol}[theorem]{Corollary}

\theoremstyle{definition}

\newtheorem{rmk}[theorem]{Remark}

\newtheorem{defin}[theorem]{Definition}

\begin{document}

\title{\bf\Large Existence, nonexistence and uniqueness for Lane-Emden type fully nonlinear systems}
\author[1]{Liliane Maia\footnote{l.a.maia@mat.unb.br}}
\author[2]{Gabrielle Nornberg\footnote{gabrielle@icmc.usp.br}}
\author[3]{Filomena Pacella\footnote{pacella@mat.uniroma1.it}}
\affil[1]{\small Universidade de Brasília, Brazil}
\affil[2]{\small Instituto de Ciências Matemáticas e de Computação, Universidade de São Paulo, Brazil}
\affil[3]{\small Sapienza Università di Roma, Italy}

{\date{\today}}

\maketitle

{\small\noindent{\bf{Abstract.}} 
We study existence, nonexistence, and uniqueness of positive radial solutions for a class of
nonlinear systems driven by Pucci extremal operators under a Lane-Emden coupling configuration.
Our results are based on the analysis of the associated quadratic dynamical system and energy methods. For both regular and exterior domain radial solutions we obtain new regions of existence and nonexistence. Besides, we show an exclusion principle for regular solutions, either in $\rN$ or in a ball, by exploiting the uniqueness of trajectories produced by the flow.

In particular, for the standard Lane-Emden system involving the Laplacian operator, we prove that the critical hyperbola of regular radial positive solutions is also the threshold for existence and nonexistence of radial exterior domain solutions with Neumann boundary condition. As a byproduct, singular solutions with fast decay at infinity are also found. 
}
\medskip

{\small\noindent{\bf{Keywords.}} {Fully nonlinear systems; Liouville properties; existence of positive solutions; uniqueness; dynamical system.}

\medskip

{\small\noindent{\bf{MSC2020.}} {35J15, 35J60, 35B09, 34A34.}

\section{Introduction and main results}\label{Introduction}

In this paper we study existence, nonexistence, and uniqueness of positive radial solutions of fully nonlinear elliptic partial differential systems of the following type
\begin{align}\label{LE} 
\left\{
\begin{array}{rclcl}
\mathcal{M}^\pm_{\lambda,\Lambda} (D^2 u)+ v^p &=&0 &\mbox{in} & \;\Omega \\
\mathcal{M}^\pm_{\lambda,\Lambda} (D^2 v)+ u^q &=&0 &\mbox{in} & \;\Omega \\
u,v&>& 0 &\mbox{in} & \;\Omega ,
\end{array}
\right.
\end{align}
in the superlinear regime $pq>1$, for $p,q>0$, and $\Omega\subset \rN$, $N\geq 3$.
Here, $\mathcal{M}^\pm_{\lambda,\Lambda}$ are the Pucci's extremal operators, for $0<\lambda\le \Lambda$, which play an essential role in stochastic control theory and mean field games. 

We deal with classical solutions of \eqref{LE} that are $C^2$ for $r>0$.
We call a solution pair $(u,v)$ \textit{regular} if $u$ and $v$, together with their derivatives, are continuously defined at $x=0$. 

For regular solutions, the set $\Omega$ is either the whole space $\rN$, or a ball $B_R$ of radius $R>0$ centered at the origin, or an exterior domain $\rN\setminus B_R$.
In the case of singular solutions, $\Omega$ can be either $\rN\setminus\{0\}$ or $B_R\setminus\{0\}$, and we assume the condition
\begin{align}\label{H singular} \textstyle{\lim_{r\to 0} \,u(r) = \lim_{r\to 0} \,v(r) =+\infty , \;\;\; r=|x|.}
\end{align}
In addition, whenever $\Omega$ has a boundary, we prescribe the Dirichlet condition 
\begin{align}\label{H Dirichlet}
u,v=0 \textrm{ on } \partial\Omega , \quad \textrm{ or } \quad u,v=0 \textrm{ on $\partial\Omega\setminus\{0\}$\, under } \eqref{H singular}.
\end{align}
We highlight that positive solutions of \eqref{LE} in a ball  for $pq>1$ are always radial, see \cite{EG} (see also \cite{BdaLioSym} for the respective scalar case $p=q$).

Next, we define the positive parameters $\alpha$, $\beta$ given by
\begin{align}\label{def alpha, beta}
{\alpha =\frac{2(p+1)}{pq-1}, \qquad \beta=\frac{2(q+1)}{pq-1}}
\end{align}
for $p,q>0$ such that $pq>1$. They play a role in the scaling
 \begin{align}\label{eq scaling}
 \textrm{	${u}_\gamma (r)=\gamma^\alpha u(\gamma r)$, \qquad ${v}_\gamma (r)=\gamma^\beta v(\gamma r)$, \qquad for \,$\gamma>0$,}
 \end{align}
under which the system \eqref{LE} in $\rN$ is invariant; see also Remark \ref{rescaling}.

Let us have in mind the following initial value problem with positive shooting parameters $\xi,\eta$, which   produces the radial regular solutions of \eqref{LE},
\begin{align}\label{shooting}
\begin{cases}
u^{\prime\prime}\;=\; M_\pm\left( -r^{-1}(N-1)\, m_\pm(u^\prime)-  |v|^{p-1}v \right),
\quad u (0)=\xi  , \;\; u^\prime (0)=0 , \qquad \xi>0 ,\\
v^{\prime\prime}\;=\; M_\pm\left( -r^{-1}(N-1)\, m_\pm(v^\prime)-  |u|^{q-1}u \right),
\quad v (0)=\eta  , \;\; v^\prime (0)=0 , \qquad \eta>0 ,
\end{cases}
\end{align}
where $M_\pm$ and $m_\pm$ are the Lipschitz functions
\begin{align}\label{m,M+}
m_+(s)=
\begin{cases}
\lambda s\; \textrm{ if } s\leq 0 \\
\Lambda s\; \textrm{ if } s> 0
\end{cases}\;
\textrm{and}\quad
M_+(s)=
\begin{cases}
s/\lambda\; \textrm{ if } s\leq 0 \\
s/ \Lambda\; \textrm{ if } s> 0;
\end{cases}
\end{align}
\vspace{-0.6cm}
\begin{align}\label{m,M-}
m_-(s)=
\begin{cases}
\Lambda s\; \textrm{ if } s\leq 0 \\
\lambda s\; \textrm{ if } s> 0
\end{cases}\;
\textrm{and}\quad
M_-(s)=
\begin{cases}
s/\Lambda\; \textrm{ if } s\leq 0 \\
s/ \lambda\; \textrm{ if } s> 0.
\end{cases}
\end{align}

The first main result of this paper concerns uniqueness of regular radial solutions of \eqref{LE} when $\Omega=\rN$, and uniqueness of solutions to the associated Dirichlet problem \eqref{LE}, \eqref{H Dirichlet} for $\Omega=B_R$.

\begin{theorem}\label{teo uniqueness}
	Let $p,q>0$ with $pq > 1$. Then:
	\begin{enumerate}[(i)]
		\item problem \eqref{LE} with $\Omega=\rN$ has at most one radial positive regular solution up to scaling \eqref{eq scaling}. Moreover, the set of shooting parameters $(\xi,\eta)$ for which \eqref{shooting} admits a positive radial solution in $\rN$ is the graph of a simple smooth curve $\eta=c\, \xi^{\frac{q+1}{p+1}}$,	where $c$ is a constant;
		
		\item problem \eqref{LE}, \eqref{H Dirichlet} with $\Omega=B_R $, $R>0$, has at most one positive solution $(u,v)$, which is radial. Further, given $\bar R>0$, any other solution pair $(\bar{u},\bar{v})$ in  $B_{\bar R}$ is obtained from $(u,v)$ by rescaling, i.e.\ $(\bar{u}, \bar{v})=( u_\gamma, v_\gamma)$ in \eqref{eq scaling}, for some $\gamma>0$ such that $\bar R=\gamma R$.
	\end{enumerate} 
\end{theorem}

Next we consider the following hyperbolas $\sH$ and $\tH$ on the plane,
\begin{equation}\label{sH}
(p,q)\in\sH\quad\Leftrightarrow\quad {\alpha+\beta=N-2 \quad\Leftrightarrow\quad \frac{1}{p+1}+\frac{1}{q+1}=\frac{N-2}{N},}
\end{equation}
\begin{equation}\label{tH}
(p,q)\in\tH_\pm\quad\Leftrightarrow\quad{\alpha+\beta=\tilde{N}_\pm-2 \quad\Leftrightarrow\quad \frac{1}{p+1}+\frac{1}{q+1}=\frac{\tilde{N}_\pm-2}{\tilde{N}_\pm},}
\end{equation}
where $\tilde{N}_\pm$ are the dimensional-like numbers
\begin{align}\label{dim-like}
\textstyle{\tilde{N}_+=\frac{\lambda}{\Lambda}(N-1)+1, \qquad \tilde{N}_-=\frac{\Lambda}{\lambda}(N-1)+1.}
\end{align} 
For a pair $(p,q)$ with $p,q>0$ and $pq>1$, the region below $\sH$ is expressed by $\alpha +\beta >N-2$, and the region above $\sH$ by $\alpha +\beta<N-2$; the same for $\tH_\pm$ by replacing $N$ by $\tilde{N}_\pm$.

\begin{defin}\label{def fast}
	Let $(u,v)$ be a radial solution pair of \eqref{LE} defined for all $r\ge r_0$, for some $r_0\ge 0$, $r=|x|$. We say $(u,v)$ is \textit{fast decaying} if there exists $c>0$ such that at least one between $u,v$ satisfies $\lim_{r\to\infty} r^{\tilde{N}-2} w(r) = c$, where $\tilde{N}$ is either $\tilde{N}_+$ if the operator is $\mathcal{M}^+$ or $\tilde{N}_-$ for $\mathcal{M}^-$.
\end{defin}

In the case of regular solutions, for the standard Lane-Emden system involving the Laplacian operator $\lambda=\Lambda=1$, the following result is well known.

\begin{theorem}[\cite{BV, Dalmasso, Mitidieri, SZcpde, SZnonexist}]\label{Th Lapl known}
Below $\sH$, for each $R>0$ there exists a unique  radial solution of \eqref{LE}, \eqref{H Dirichlet} in $B_R$, and there is no  radial solution of \eqref{LE} in $\rN$.
On $\sH$ there exists a unique radial fast decaying solution of \eqref{LE} in $\rN$ up to scaling, and there is no  radial solution of \eqref{LE}, \eqref{H Dirichlet} in any ball.
Above $\sH$ there exists a unique radial solution of \eqref{LE} in $\rN$ up to scaling, and there is no radial solution of \eqref{LE}, \eqref{H Dirichlet} in any ball. 
\end{theorem}

In what concerns the qualitative analysis of regular solutions in Theorem \ref{Th Lapl known}, we obtain the following concavity result, for both solutions in the ball and fast decaying ones. 

\begin{theorem}\label{Th concavidade Lapl}
If $\lambda=\Lambda$ and the pair $(p,q)$ is below or on the hyperbola $\sH$,
then a regular solution $(u,v)$ of \eqref{LE} is such that  $u$ and $v$ change concavity exactly once. 
\end{theorem}

The respective nonradial case has been widely investigated at least since the works of de Figueiredo, see \cite{Djaironotes81}. The nonexistence of solutions between hyperbolas $pq=1$ and $\sH$ is still open in general, except for dimensions $N=3,4$ in \cite{LE3d, LE4d}.

In the fully nonlinear scalar perspective, the very first classification result on radial positive solutions of Lane-Emden equations involving Pucci operators was obtained by Felmer and Quaas in \cite{FQaihp, FQind}, by using an innovative ODE approach. Further existence in annuli and exterior domains were found in \cite{GLPradial} and \cite{GILexterior2019}, respectively. More recently  in \cite{MNPscalar} we derived a complete classification of singular solutions, even for weighted equations, via a special change of variables, as in \cite{BV}, that permits us visualizing all orbits of the corresponding dynamical system. We note that the study of quadratic systems to treat Emden-Fowler type problems has long been used, see also \cite{Jones1953, Wong75}.

As far as the fully nonlinear system is concerned, Quaas and Sirakov proved in \cite{BQind} that problem \eqref{LE} has no solutions in $\rN$ if at least one between $\alpha$ and $\beta$ is larger than or equal to $\tilde{N}_\pm-2$, whenever $p,q$ enjoy the superlinearity $pq>1$ under the additional assumption $p,q\ge 1$. They also used this result and the blow-up method to show existence of a solution to \eqref{LE} in a ball. In \cite{ABsupersolutions}, Armstrong and Sirakov removed the assumption $p,q\ge 1$, by extending the result to the region 
\begin{align}\label{region ABQ}
\mathcal{R}_{s}^\pm=\{ (p,q)\in \real^2: \; p,q>0, \; pq>1, \; \textrm{with } \alpha\ge \tilde{N}_\pm -2 \textrm{ or } \beta\ge \tilde{N}_\pm -2\},
\end{align}
on the plane $(p,q)$.
We recall \eqref{region ABQ} describes the range of nonexistence of supersolutions, which in turn extends the scalar case developed in \cite{CutriLeoni}. 
Obviously this is far from optimal when referring to solutions of \eqref{LE}.

In this paper we improve such region in what concerns nonexistence of solutions to \eqref{LE}, as long as $p,q$  are close to the diagonal $p=q$, see Figure \ref{Fig BQ-nosso}.
In the case of the operator $\M^+$, the novelty depends on whether $\mathcal{R}_{s}^+$ does not contain $\sH$. This would correspond to the scalar situation $p_s^+<p_\Delta$. 
Meanwhile, for the operator $\M^-$ this is always an improvement, in analogy to the scalar case $p_s^-<p_\Delta$. Here,
\begin{center}
	 $p_s^\pm=\frac{\tilde{N}_\pm}{\tilde{N}_\pm-2}$\quad and \quad $p_\Delta=\frac{N+2}{N-2}$,
\end{center}
where $p_s^\pm$ is the corresponding Serrin exponent for the Pucci's operator $\M^\pm$ from \cite{CutriLeoni}, while $p_\Delta$ is the critical exponent for the Lane-Emden equation driven by the Laplacian operator.

\smallskip

Now we define the regions, down and up, for the operator $\M^+$,
\begin{align}\label{Rd}
\textstyle \mathcal{R}_d^+=\{\,(p,q)\in \real^2_+:\;\,  \,\frac{N}{p+1}+\frac{\tilde{N}_+}{q+1}> N-2,\;\;\;
\frac{\tilde{N}_+}{p+1}+\frac{N}{q+1}> N-2,\;\;\,pq>1\,\},
\end{align}
\vspace{-0.3cm}
\begin{align}\label{Ru}
\textstyle \mathcal{R}_u^+=\{\,(p,q)\in \real^2_+:\;\,  \,\frac{N}{p+1}+\frac{\tilde{N}_+}{q+1}<  \tilde{N}_+-2,\;\;\;
\frac{\tilde{N}_+}{p+1}+\frac{N}{q+1}< \tilde{N}_+-2\,\},
\end{align}
and the respective lower and upper regions for $\M^-$ given by
\begin{align}
\textstyle \mathcal{R}_d^-=\{\,(p,q)\in \real^2_+:\;\,  \,\frac{N}{p+1}+\frac{\tilde{N}_-}{q+1}> \tilde{N}_--2,\;\;\;
\frac{\tilde{N}_-}{p+1}+\frac{N}{q+1}> \tilde{N}_--2,\;\;\,pq>1\,\},\label{Rd-}
\end{align}
\vspace{-0.3cm}
\begin{align}
\textstyle \mathcal{R}_u^-=\{\,(p,q)\in \real^2_+:\;\,  \,\frac{N}{p+1}+\frac{\tilde{N}_-}{q+1}< N-2,\;\;\;
\frac{\tilde{N}_-}{p+1}+\frac{N}{q+1}< {N}-2\,\}. \label{Ru-}
\end{align}

Our next main result exhibits existence and nonexistence results in these regions.

\begin{figure}[!htb]\centering	\includegraphics[scale=0.33]{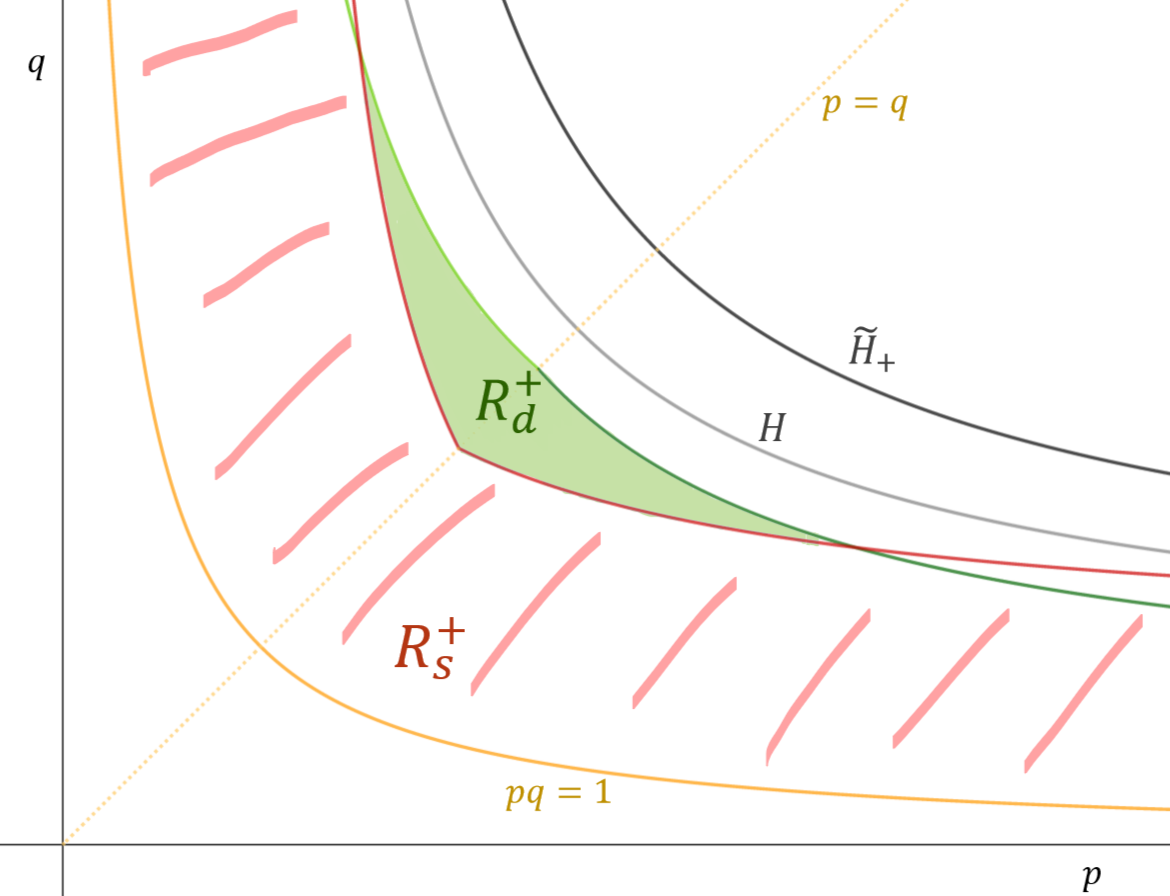}	
	\caption{Our improved region $\mathcal{R}_d^+$ in \eqref{Rd} for $\M^+$ with respect to region $\mathcal{R}_{s}^+$ in \eqref{region ABQ}.
	}\label{Fig BQ-nosso}
\end{figure}

\begin{theorem}[Regular solutions]\label{Th regular} Let $\lambda\neq \Lambda$.
With respect to regular solutions of \eqref{LE}, it follows:
\begin{enumerate}[(i)]
	\item if $(p,q)\in \mathcal{R}_d^\pm$ then problem \eqref{LE} in $\rN$ does not have positive radial solutions. Moreover, for each $R>0$ there exists a unique positive solution \eqref{LE}, \eqref{H Dirichlet} in the ball $B_R$;
		
	\item if $(p,q)\in\overline{\mathcal{R}_u^\pm}$ then there exists a unique (up to scaling) positive radial solution of \eqref{LE} in $\rN$. Further, there is no solutions of \eqref{LE}, \eqref{H Dirichlet} in any ball $B_R$.
\end{enumerate}
\end{theorem}

\begin{figure}[!htb]\centering	\includegraphics[scale=0.4]{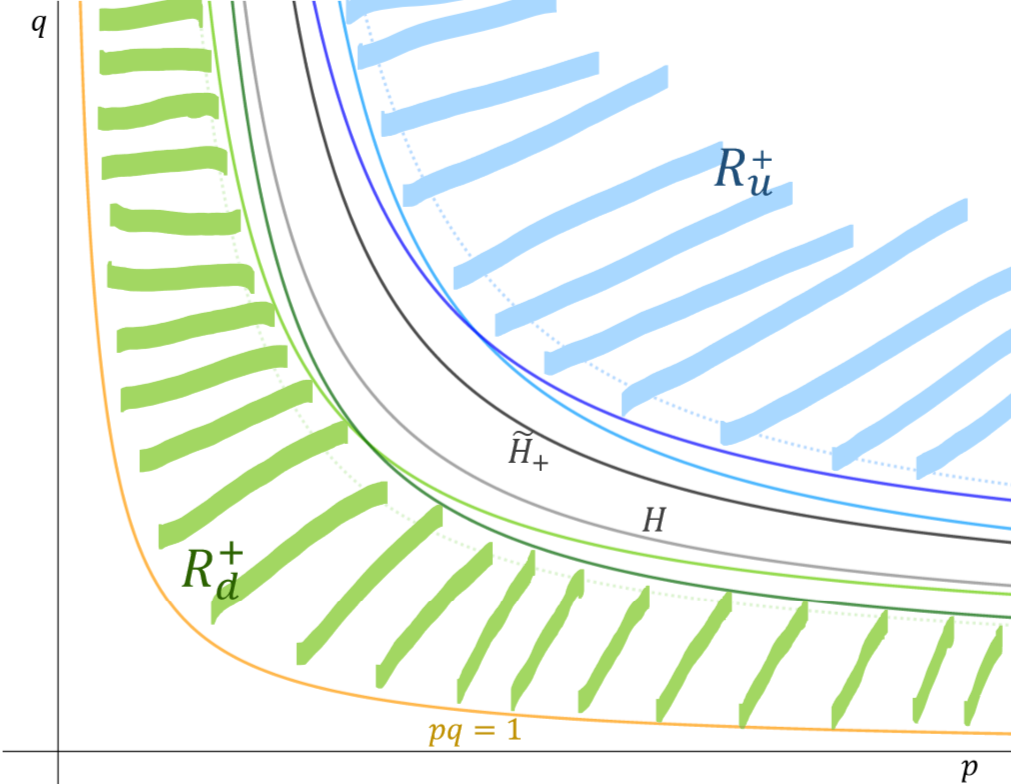}	
	\caption{The regions $\mathcal{R}_d^+$ in \eqref{Rd} and $\mathcal{R}_u^+$ in \eqref{Ru} for the operator $\M^+$. 
	}\label{Fig hyp2}
\end{figure}

Our approach for proving Theorem \ref{Th regular} relies on a suitable choice of piecewise defined energies for the associated quadratic system. They are discontinuous functions ruled by four hyperbolas rather than just one when comparing it to proof of Theorem \ref{Th Lapl known} in \cite{BV}.
Moreover, regarding the study of fast decaying solutions, we develop a brand new study at the stationary points responsible for generating fast decaying trajectories in the dynamical system.
It is an open question whether or not uniqueness of fast decaying solutions holds in the region above $\mathcal{R}_s^\pm$ in light of Theorem \ref{teo uniqueness}, cf.\ the scalar case \cite{{GILexterior2019}, MNPscalar}.

\smallskip

We define the set of regular solutions in a ball as
\begin{center}
	$\C= \{(p,q)\in \real^2: \, p,q>0, \, pq>1$; there exists a solution $(u,v)$ of \eqref{LE}, \eqref{H Dirichlet} in $B_R\}$.
\end{center}
Since $\C$ is nonempty, and the set $\{ q: (p,q)\in \C\}$ is bounded from above by Theorem \ref{Th regular} for each fixed $p$, we may define $\partial \C$. 
We believe $\C$ is connected and $\partial \C$ consists of a continuous curve on the plane $(p,q)$.
We show in Corollary \ref{cor change concav 1x Rd} that regular solutions in $\mathcal{R}_d^\pm$ change concavity exactly once, although we guess this property is preserved up to $(p,q)\in \overline{\C}$ as in the scalar case \cite{MNPscalar}.

Next, we denote the set of \textit{ground state} regular solutions as
\begin{center}
	$\mathcal{G}= \{(p,q)\in \real^2: \, p,q>0, \,pq>1$; there exists a radial solution $(u,v)$ of \eqref{LE} in $\rN\}$.
\end{center}

Our next result allows us to characterize the complementary set to $\mathcal{G}$ as $\C$.
In other words, it says that we cannot have regular solutions both in a ball and in $\rN$ simultaneously.

\begin{theorem}[Exclusion principle]\label{teo uniqueness trajectory} For $\C$ and $\mathcal{G}$ as above concerning regular solutions, we have	
$ \{(p,q)\in \real^2: \; p,q>0, \; pq >1 \}\, =\,\C \sqcup \mathcal{G}$, 
where $\sqcup$ is a disjoint union.
\end{theorem}

Our strategy for obtaining Theorem \ref{teo uniqueness trajectory} is to show how uniqueness of regular solutions presented in Theorem \ref{teo uniqueness} either in $\rN$ or in $B_R$, up to rescaling, is translated into uniqueness of trajectories for the respective dynamical system.  This is the heart of the paper and, up to our knowledge, it is the first time an exclusion result of this nature is proven in this context.
\vspace{0.05cm}

In view of Theorem \ref{teo uniqueness trajectory}, we conjecture that $\partial\C$ turns out to be a critical curve on the $(p,q)$ plane, being the threshold between existence and nonexistence of regular solutions in $\rN$. 
As a byproduct, we also conjecture that this produces a critical curve for existence and nonexistence of fast decaying exterior domain solutions.
We point out that $\partial\C \subset \mathcal{G}$ since $\C$ is open (see Proposition \ref{C is open}).
It is natural to expect, on this critical curve, the behavior of the solutions at infinity being subject to a fast decaying profile.

We also give a result on exterior domain and singular solutions. Up to our knowledge, they are novelties even for standard Lane-Emden systems driven by the Laplacian operator.

Particularly for the next theorem, we consider the (nonempty) regions $\mathcal{R}_D^\pm \subset \mathcal{R}_d^\pm$ given by
\begin{center}
\;\,$\mathcal{R}_D^+:=\{\, (p,q)\in \real^2: \; \frac{1}{p+1}+\frac{1}{q+1}>   \frac{\Lambda}{\lambda} \frac{2N-\tilde N_+-2}{\tilde N_+}\,, \;\;pq>1\,\}$ \;\; for $\M^+$, \smallskip

$\mathcal{R}_D^-:=\{\, (p,q)\in \real^2: \; \frac{1}{p+1}+\frac{1}{q+1}>   \frac{\Lambda}{\lambda} \frac{2\tilde N_--N-2}{\tilde N_-}\,, \;\;pq>1\,\}$ \;\; for $\M^-$.
\end{center}
\begin{theorem}[Exterior domain Neumann]\label{Th exterior} Let $\lambda\le  \Lambda$ and $R>0$.
Regarding solutions of \eqref{LE} defined in the exterior domain $\rN\setminus B_R$, with $u,v>0$ on $\partial B_R$, and $\partial_\nu u, \partial_\nu v=0$ on $\partial B_R$, it holds:
\begin{enumerate}[(i)]
\item if $ (p,q)\in \overline{\mathcal{R}_D^\pm}$ then there is no radial positive exterior domain solution of \eqref{LE};
			
\item if $(p,q)\in \mathcal{R}_u^\pm$ there exist exterior domain fast decaying solutions of \eqref{LE}.
\end{enumerate}
Moreover, if $\lambda=\Lambda$ then the hyperbola $\sH$ in \eqref{sH} gives us the threshold for existence and nonexistence of exterior domain solutions with Neumann boundary condition.
\end{theorem}

\begin{theorem}[Singular solutions]\label{Th singular} Let $\lambda\le  \Lambda$.
	Regarding singular solutions of \eqref{LE}--\eqref{H singular}, one has:
	\begin{enumerate}[(i)]
		\item if $(p,q)\in\mathcal{R}_d^\pm$ there exists a singular fast decaying solution of \eqref{LE}, \eqref{H singular} in $\rN$;
		
		\item if $(p,q)\in \overline{\mathcal{R}_u^\pm}$ then there is no singular fast decaying solution of \eqref{LE}, \eqref{H singular} in $\rN$.
	\end{enumerate}
Furthermore,  if $\lambda=\Lambda$ then the hyperbola $\sH$ in \eqref{sH} divides existence and nonexistence of fast decaying singular solutions in $\rN$.
\end{theorem}

Finally, for the standard Lane-Emden system, we obtain the following Liouville type result on exterior domain solutions with Dirichlet boundary condition.
\begin{theorem}[Liouville exterior domain Dirichlet]\label{Th exterior Dir} Let $\lambda= \Lambda$.
	If $\alpha+\beta \ge N-2$
	then there is no radial positive exterior domain solution of \eqref{LE} for any $R>0$ with Dirichlet boundary condition \eqref{H Dirichlet}. 
\end{theorem}

We highlight that the extension of our results to Hénon weights $|x|^a, |x|^b$ with $a,b> -1$ could be treated as in \cite{MNPscalar}, in the spirit of \cite{BV}. We prefer to skip it in this work to keep the presentation simpler, by focusing in what is new for the system.

\smallskip

The paper is organized as follows. In Section \ref{Preliminaries} we recall some preliminary facts on radial solutions of Pucci's operators and study the associated quadratic system. Section \ref{section global} is dedicated to the global study of the dynamical system and we prove Theorems  \ref{teo uniqueness} and \ref{teo uniqueness trajectory}. Setion~\ref{section.1 energy} is devoted to energy and qualitative analyses, and the proof of the remaining theorems are provided.

\section{The dynamical system}\label{Preliminaries}

In this section we define some new variables which allow us to transform the radial fully nonlinear equations into a quadratic dynamical system.

\subsection{The second order PDE system}

We start by recalling that the Pucci's extremal operators $\mathcal{M}^\pm_{\lambda,\Lambda}$, for $0<\lambda\leq \Lambda$, are defined as
$$
\textstyle{\mathcal{M}^+_{\lambda,\Lambda}(X):=\sup_{\lambda I\leq A\leq \Lambda I} \mathrm{tr} (AX)\,,\quad \mathcal{M}^-_{\lambda,\Lambda}(X):=\inf_{\lambda I\leq A\leq \Lambda I} \mathrm{tr} (AX),}
$$
where $A,X$ are $N\times N$ symmetric matrices, and $I$ is the identity matrix.
Equivalently, if we denote by $\{e_i\}_{1\leq i \leq N}$ the eigenvalues of $X$, we can define the Pucci's operators as
\begin{align}\label{def Pucci}
\textstyle{\textrm{$\mathcal{M}_{\lambda,\Lambda}^+(X)=\Lambda \sum_{e_i>0} e_i +\lambda \sum_{e_i<0} e_i$, \;\;\; $\mathcal{M}_{\lambda,\Lambda}^-(X)=\lambda \sum_{e_i>0} e_i +\Lambda \sum_{e_i<0} e_i$}. }
\end{align}
From now on we will drop writing the parameters $\lambda,\Lambda$ in the notations for the Pucci's operators.

When $u$ is a radial function, to simplify notation we set $u(|x|)=u (r)$ for $r=|x|$. If in addition $u$ is $C^2$, the eigenvalues of the Hessian matrix $D^2 u$ are $u^{\prime\prime}$ which is simple, and $\frac{u^\prime (r)}{r}$ with multiplicity $N-1$.  

The Lane-Emden system \eqref{LE} for $\mathcal{M}^+$ is written in radial coordinates as
\begin{align}\label{P radial m}\tag{$P_+$}
\left\{
\begin{array}{l}
u^{\prime\prime}\;=\; M_+(-r^{-1}(N-1)\, m_+(u^\prime)- v^p ), \\
v^{\prime\prime}\;=\; M_+(-r^{-1}(N-1)\, m_+(v^\prime)- u^q ) , \quad u,\, v> 0  ,
\end{array}
\right.
\end{align}
while for $\M^-$ one has
\begin{align}\label{P radial m-}\tag{$P_-$}
\left\{
\begin{array}{l}
u^{\prime\prime}\;=\; M_-(-r^{-1}(N-1)\, m_-(u^\prime)- v^p ), \\
v^{\prime\prime}\;=\; M_-(-r^{-1}(N-1)\, m_-(v^\prime)- u^q ) , \quad u,\, v> 0  ,
\end{array}
\right.
\end{align}
which are understood in the maximal interval where $u,v$ are both positive.

We stress that by \textit{ regular solution} of \eqref{P radial m} or \eqref{P radial m-} we mean a solution pair $(u,v)$ which is positively defined at $r=0$, and twice differentiable up to $0$.

Let us consider the following functions which determine the sign of $-u^{\prime\prime}, -v^{\prime\prime}$,  
\begin{align}\label{def H1, H2}
H_1(r)= r^{-1}  (N-1) m_\pm(u^\prime) +v^p , \qquad H_2(r)= r^{-1} (N-1) m_\pm(v^\prime)+u^q.
\end{align}

\smallskip

Next we show that solutions of \eqref{P radial m} or \eqref{P radial m-} are strictly concave around $r=0$.
\begin{lem}\label{lemma u,v concave at 0}
	Every regular solution pair $(u,v)$ of \eqref{P radial m} or \eqref{P radial m-} satisfies $u^{\prime\prime}(0)< 0$, $v^{\prime\prime}(0)<0$. In particular, $u$ and $v$ are both concave in a neighborhood of $r=0$.
\end{lem}

\begin{proof}
	Suppose by contradiction that $u^{\prime\prime} (0)\geq 0$. Then note that 
	$$ \textstyle
H_1(0)=\lim_{r\rightarrow 0^+} H_1(r)
	=\lim_{r\rightarrow 0^+} \left\{ (N-1)\, m_\pm\left(\frac{u^\prime (r) -u^\prime (0) }{r}\right)+v^p(r) \right\}\geq v^p(0)>0,
	$$
	since $m_\pm$ are Lipschitz continuous. Now the continuity of $M_\pm$ yields
	$u^{\prime\prime}(0)=\lim_{r\rightarrow 0^+} u^{\prime\prime}(r)=  M_\pm(-H_1(0))<0$, a contradiction. Analogously one shows that $v^{\prime\prime}(0)<0$.
\end{proof}

\begin{lem}\label{lemma u,v decreasing}
If $(u,v)$ is a regular or singular solution pair of \eqref{P radial m} or \eqref{P radial m-} (together with \eqref{H singular} in the singular case) then $u^\prime, v^\prime <0$ in $(0,+\infty)$ as long as both $u,v$ remain positive. 
\end{lem}

\begin{proof}
Let us prove only the monotonicity for $u$, since the one for $v$ is analogous.
If $(u,v)$ is regular, by Lemma \ref{lemma u,v concave at 0} we have $u^{\prime\prime}(0)<0$, thus $u^\prime$ is decreasing in a neighborhood of $0$. Since $u^\prime (0)=0$, then $u^\prime (r) <0$ in some interval of positive $r$.
On the other hand, if $(u,v)$ is singular satisfying \eqref{H singular}, then $u$ is necessarily decreasing in a neighborhood of $r=0$. 
Anyway, let $(0,r_0)$ be the maximal interval where $u^\prime <0$. There are two possibilities: either $r_0=\infty$, or $u^\prime (r_0)=0$.

If we had $u^\prime (r_0)=0$ at some point where $u$ is positive, i.e.\ with $u(r_0)>0$, we would obtain $u^{\prime\prime} (r_0)<0$ from the equation. That is, $u^\prime$ is strictly decreasing when passing through the point $r_0$, and hence attains negatives values on the left of $r_0$. This contradicts the definition of $r_0$. Thus $u^\prime$ can never vanish at a positivity point of $u$ if $u$ solves either \eqref{P radial m} or \eqref{P radial m-}. 
\end{proof}

\medskip

In what concerns the initial value problem \eqref{shooting}, we recall once again that a regular solution $(u,v)$ of \eqref{shooting} is twice differentiable up to $0$. 
Since a solution pair of \eqref{shooting} is positive near $0$, by the previous lemmas we obtain that it is a solution for the system driven by the Laplacian operator around $r=0$. Now, by \cite[Lemma 2.1]{SZcpde} for instance, local solutions of \eqref{shooting} exist.
We denote by $u_{\xi,\eta}$, $v_{\xi,\eta}$ the solutions of \eqref{shooting}.
Then we set $R_{\xi,\eta}$, with $R_{\xi,\eta}\leq +\infty$, the radius of the maximal interval $[0,R)$ where $u_{\xi,\eta}$ and $v_{\xi,\eta}$ are both positive.

\vspace{0.03cm}

Hence $(u,v)$ is a solution of $(P_\pm)$ in $[0,R_{\xi,\eta})$. 
Obviously, if $R_{\xi,\eta}=+\infty$ then $(u,v)$ corresponds to a radial positive solution of \eqref{LE} for $\Omega=\rN$.
When $R_{\xi,\eta}<+\infty$ and $u(R_{\xi,\eta})=v(R_{\xi,\eta})=0$ it gives a positive solution of the Dirichlet problem \eqref{LE}, \eqref{H Dirichlet} in the ball $\Omega= B_{R_{\xi,\eta}}$. 

\begin{rmk}\label{rescaling} Given a regular positive solution pair $(u,v)$  in $[0,R_{\xi,\eta})$, with $u=u_{\xi,\eta}$ and $v=v_{\xi,\eta}$, which satisfies \eqref{shooting} for some positive constants $\xi,\eta$, then the rescaled functions  ${u}_\gamma$ and ${v}_\gamma$ as in \eqref{eq scaling}, $\gamma>0$, still give a positive solution pair of the same equation in $[0, \gamma^{-1}{R_{\xi,\eta}})$ with initial values $u_\gamma(0)=\gamma^\alpha \xi$ and $v_\gamma(0)=\gamma^\beta \eta$.
	
	If $u,v$ are defined in the whole interval $[0,+\infty)$, thus there is a family of entire regular solutions obtained via $u_\gamma$, $v_\gamma$ for all $\gamma>0$. If there are no other entire solutions then we say that $(u,v)$ is \textit{unique up to scaling}.
	
	On the other hand, a solution in the ball of radius $R_{\xi,\eta}$ automatically produces a solution for an arbitrary ball, by properly choosing the parameter $\gamma>0$.
\end{rmk}

\subsection{The associated quadratic system}\label{section 2.1}

Let $u,v$ be a positive solution pair of \eqref{P radial m} or \eqref{P radial m-}, then we can define the new functions
\begin{align}\label{X,Y,Z,W sem a}
X(t)=-\frac{ru^\prime}{u}, \quad Y(t)=-\frac{rv^\prime}{v}, \quad Z(t)=-\frac{r v^p}{u^\prime}, \quad W(t)=-\frac{r u^q}{v^\prime}, 
\end{align}
for $t=\mathrm{ln}(r)$, whenever $r>0$ is such that $u,v\neq 0$ and $u^\prime ,v^\prime\neq  0$.
The phase space is contained in $ \real^4$ and,
throughout the paper, we denote its positive cone as 
\begin{center}
	$\K=\{\,(X,Y,Z,W)\in \real^4:\; X,Y,Z,W>0\,\}$.
\end{center}

Since we are studying positive solutions, the points $(X(t),Y(t),Z(t),W(t))$ belong to $\K$ when both $u^\prime,v^\prime <0$.
As a consequence of this monotonicity, the problems \eqref{P radial m} and \eqref{P radial m-} then become in $\K$ as:
\begin{align}\label{P M+ 1Q}
\textrm{for $\M^+$ in $\K$} :\quad
\left\{
\begin{array}{l}
u^{\prime\prime}\;=\; M_+(-\lambda r^{-1}(N-1)\, u^\prime- v^p ), \\
v^{\prime\prime}\;=\; M_+(-\lambda r^{-1}(N-1)\,v^\prime- u^q ) , \quad u,\, v> 0  ;
\end{array}
\right.
\end{align}
\vspace{-0.5cm}
\begin{align}\label{P M- 1Q}
\textrm{\;\;for $\M^-$ in $\K$} :\quad \left\{
\begin{array}{l}
u^{\prime\prime}\;=\; M_-(-\Lambda r^{-1}(N-1)\, u^\prime- v^p ), \\
v^{\prime\prime}\;=\; M_-(-\Lambda r^{-1}(N-1)\,v^\prime- u^q ) , \quad u,\, v> 0  .
\end{array}
\right.
\end{align}

\begin{rmk}\label{Rmk sing reg}
Regular or singular solutions of \eqref{P radial m} and \eqref{P radial m-} enjoy the monotonicity $u^{\prime},v^{\prime}<0$ by Lemma \ref{lemma u,v decreasing}. Thus their study and respective dynamics is restricted to $\overline{\K}.$ 
\end{rmk}

In terms of the functions \eqref{X,Y,Z,W sem a}, we derive the following autonomous dynamical system, corresponding to \eqref{P M+ 1Q} for $\M^+$, where the dot $\dot{ }$ stands for $\frac{\mathrm{d}}{\mathrm{d} t}$,
\begin{align}\label{DS+}
\textrm{$\M^+$ in $\K$ : \quad}
\left\{
\begin{array}{ccl}
\dot{X} &=& \;X \; [\,X+1-M_+ (\lambda(N-1)-Z )\,] \\
\dot{Y} &=& \;Y\; [\,Y+1-M_+ (\lambda(N-1)-W )\,] \\
\dot{Z} &=& \;Z\; [\,1-pY +M_+ (\lambda(N-1)-Z )\,] \\
\dot{W} &=& W\; [\,1-qX +M_+ (\lambda(N-1)-W )\,].
\end{array}
\right.
\end{align}
Similarly one has for $\M^-$, associated to \eqref{P M- 1Q},
\begin{align}\label{DS-}
\textrm{$\M^-$ in $\K$ : \quad}
\left\{
\begin{array}{ccl}
\dot{X} &=& \;X \; [\,X+1-M_- (\Lambda(N-1)-Z )\,] \\
\dot{Y} &=& \;Y\; [\,Y+1-M_- (\Lambda(N-1)-W )\,] \\
\dot{Z} &=& \;Z\; [\,1-pY +M_- (\Lambda(N-1)-Z )\,] \\
\dot{W} &=& W\; [\,1-qX +M_- (\Lambda(N-1)-W )\,].
\end{array}
\right.
\end{align}
We stress that trajectories of \eqref{DS+}, \eqref{DS-} correspond to positive, decreasing solutions of \eqref{P radial m}, \eqref{P radial m-}. 
\smallskip

On the other hand, given a trajectory $\tau=(X,Y,Z,W)$ of \eqref{DS+} or \eqref{DS-} in $\K$, we define
\begin{align}\label{def u via X,Z}
\textstyle{ u(r)=r^{-\alpha} (XZ)^{\frac{1}{pq-1}}(YW)^{\frac{p}{pq-1}}(t) , \;\; v(r)=r^{-\beta} (XZ)^{\frac{q}{pq-1}}(YW)^{\frac{1}{pq-1}}(t)}, \;\;\textrm{ where }\, r=e^t.
\end{align}
Thus we deduce
\begin{align*}
 u^\prime (r)& {\textstyle = -\alpha r^{-\alpha-1} (XZ)^{\frac{1}{pq-1}}(t)(YW)^{\frac{p}{pq-1}}(t) 
+ \frac{r^{-\alpha}}{pq-1} (XZ)^{\frac{1}{pq-1}-1}(t) \,\frac{\dot{X} Z+ X\dot{Z} }{r} }
\\
 &{\textstyle+ \frac{p r^{-\alpha}}{pq-1} (YW)^{\frac{p}{pq-1}-1}(t) \,\frac{\dot{Y} W+ Y\dot{W} }{r}	
= \frac{u}{r}\{ -\alpha +\frac{X+2-pY}{pq-1}+p\frac{Y+2-qX}{pq-1} \}=-\frac{X(t)u(r)}{r},
}
\end{align*}
and analogously for $v^\prime$. Since $X,Y\in C^1$, then $u,v\in C^2$. Moreover, $u,v$ satisfy either \eqref{P radial m} or \eqref{P radial m-} from the respective equations for $\dot{X},\dot{Y},\dot{Z},\dot{W}$ in the dynamical system.

\smallskip	

In other words, ($X,Y,Z,W)$ is a solution of system \eqref{DS+} or \eqref{DS-} in $\K$ if and only if $(u,v)$ defined by \eqref{def u via X,Z} is a positive pair solution of \eqref{P radial m} or \eqref{P radial m-} with $u^{\prime},v^{\prime}<0.$ 

\smallskip

An important role in the study of our problem is played by the following hyperplanes for $\M^+$,
\begin{align}\label{pi M+}
\pi_{\lambda,Z}=\{(X,Y,Z,W): Z=\lambda (N-1)\}\cap\K,\;\;\pi_{\lambda, W}=\{(X,Y,Z,W): W=\lambda (N-1)\}\cap\K,
\end{align}
which, as in the scalar case (see \cite{MNPscalar}), correspond to the vanishing of $u^{\prime\prime}$ and $v^{\prime\prime}$, respectively. They allow us to define the following regions 
\begin{align}\label{R+- M+}
R^+_{\lambda,Z}=\{(X,Y,Z,W)\in \K: Z>\lambda (N-1)\},\; R^-_{\lambda,Z}=\{(X,Y,Z,W)\in \K: Z<\lambda (N-1)\},\nonumber\\ 
R^+_{\lambda,W}=\{(X,Y,Z,W)\in \K: W>\lambda (N-1)\},\; R^-_{\lambda,W}=\{(X,Y,Z,W)\in \K: W<\lambda (N-1)\},
\end{align}
which represent the sets where the corresponding functions $u,v$ are concave or convex.
More precisely, $R^+_{\lambda,Z}$ and $R^+_{\lambda,W}$ are the regions of strictly concavity of $u$ and $v$ respectively, while $R^-_{\lambda,Z}$ and $R^-_{\lambda,W}$ are the regions of strictly convexity of $u$ and $v$.

The corresponding notations for the operator $\M^-$ are 
\begin{align}\label{pi M-}
\pi_{\Lambda,Z}=\{(X,Y,Z,W): Z=\Lambda (N-1)\}\cap\K,\;\;
 \pi_{\Lambda, W}=\{(X,Y,Z,W): W=\Lambda (N-1)\}\cap\K ,
\end{align}
\vspace{-0.8cm}
\begin{align}\label{R+- M-}
R^+_{\Lambda,Z}=\{(X,Y,Z,W)\in \K: Z>\Lambda (N-1)\},\;
R^-_{\Lambda,Z}=\{(X,Y,Z,W)\in \K: Z<\Lambda (N-1)\},\nonumber\\ 
R^+_{\Lambda,W}=\{(X,Y,Z,W)\in \K: W>\Lambda (N-1)\},\;
R^-_{\Lambda,W}=\{(X,Y,Z,W)\in \K: W<\Lambda (N-1)\}.
\end{align}

We recall that Hénon-Lane-Emden problems for Laplacian operators were already studied in \cite{BV} in terms of the dynamical system \eqref{DS+} in the case  $\lambda=\Lambda=1$ after the transformation \eqref{X,Y,Z,W sem a}.

At this stage it is worth observing that the systems \eqref{DS+} and \eqref{DS-} are continuous on $\pi_{\lambda,Z}$, $\pi_{\lambda,W}$. More than that, the right hand sides are locally Lipschitz functions of $X,Y,Z,W$, so the usual ODE theory applies. 
That is, one recovers existence, uniqueness, and continuity with respect to initial data as well as continuity with respect to the parameters $p,q$, whenever $u,v>0$.


\subsection{Stationary points and local analysis}\label{section local}

We start the section investigating the sets where $\dot{X}=0$, $\dot{Y}=0$, $\dot{Z}=0$, and $\dot{W}=0$. 
One writes the dynamical systems \eqref{DS+} and \eqref{DS-} as the following ODE first order autonomous equation
\begin{align}\label{Prob x=(X,Z)}
(\dot{X},\dot{Y},\dot{Z},\dot{W})=F(X,Y,Z,W), \qquad\textrm{where }\quad F:=(f_1,f_2,g_1,g_2). 
\end{align}
Firstly we recall some standard definitions from ODE theory.
\begin{defin}
	A stationary point $Q$ of \eqref{Prob x=(X,Z)} is a zero of the vector field $F$.
	If $\sigma_1, \sigma_2, \sigma_3,\sigma_4$ are the eigenvalues of the Jacobian matrix $DF(Q)$, then $Q$ is hyperbolic if all of them have nonzero real parts. If this is the case, $Q$ is a \textit{source} if $\mathrm{Re}(\sigma_i)>0$ for all $i=1,2,3,4$, and a \textit{sink} if $\mathrm{Re}(\sigma_i)<0$  for $i=1,2,3,4$; $Q$ is a \textit{saddle point} if it is hyperbolic and  $\mathrm{Re}(\sigma_i)<0<\mathrm{Re}(\sigma_j)$ for some $i\neq j$.
\end{defin}

The dynamical system is described through local stable and unstable manifolds near hyperbolic stationary points of the system \eqref{Prob x=(X,Z)}, see for instance \cite[Theorem 6.2]{PalisMelo}. Here the usual theory for autonomous systems applies.
No stationary points exist on the hyperplane subsets \eqref{pi M+}, \eqref{pi M-}, as we will see in Lemmas \ref{stationary M+} and \ref{stationary M-}.

Observe that if $X=0$ then $\dot{X}=0$, which means that a trajectory which starts on the hyperplane $X=0$ never leaves it; similarly for the others coordinate hyperplanes.

The following sets in $(X,Y,Z,W)$ play an important role for the system \eqref{DS+}, 
\begin{align}\label{pi 1,lambda}
\pi_{1,\lambda} = \{\,(X,Y,Z,W):\, Z = \Lambda (\tilde{N}_+-2)-\Lambda X\,\}\cap \K,
\end{align}
\vspace{-1cm}
\begin{align}\label{pi 2,lambda}
\pi_{2,\lambda} = \{\,(X,Y,Z,W):\, W = \Lambda (\tilde{N}_+-2)-\Lambda Y\,\}\cap \K,
\end{align}
which are the sets where $\dot{X}=0$, $X>0$, and $\dot{Y}=0$, $Y>0$ respectively. Also,
\begin{align}\label{pi 3,4,lambda}
\pi_{3,\lambda} =
{\pi_{3,\lambda}^+}\; \cup\;{\pi_{3,\lambda}^-}\,, \quad \pi_{4,\lambda} =
{\pi_{4,\lambda}^+}\; \cup\;{\pi_{4,\lambda}^-}\,
\end{align}
are the sets were $\dot{Z}=0$, $Z>0$, and $\dot{W}=0$, $W>0$ respectively, 
where
\begin{center}
$\pi_{3,\lambda}^+= \{(X,Y,Z,W): Z = \lambda (N-p Y)\}\cap R^+_{\lambda,Z}$, $\pi_{3,\lambda}^-= \{(X,Y,Z,W): Z = \Lambda (\tilde{N}_+- p Y)\}\cap R^-_{\lambda,Z}$.
\vspace{0.2cm}

$\pi_{4,\lambda}^+= \{(X,Y,Z,W): W= \lambda (N-q X)\}\cap R^+_{\lambda,W}$, $\pi_{4,\lambda}^-= \{(X,Y,Z,W): W = \Lambda (\tilde{N}_+- qX)\}\cap R^-_{\lambda,W}$.
\end{center}

Note that $\pi_{1,\lambda}$ is a hyperplane entirely contained in $R^-_{\lambda,Z}$ so that $\dot{X}>0$ in the region $R^+_{\lambda,Z}\cup \pi_{\lambda,Z}$.
Analogously, $\pi_{2,\lambda}\subset R^-_{\lambda,W}$ with $\dot{Y}>0$ in $R^+_{\lambda,W}\cup \pi_{\lambda,W}$.
In turn, $\pi_{3,\lambda}$ and $\pi_{4,\lambda}$ in \eqref{pi 3,4,lambda} are unions of half hyperplanes which join on $(X,\frac{1}{p},\lambda (N-1),W)\in \pi_{\lambda,W}\cap\overline{\pi}_{4,\lambda}$; and on $(\frac{1}{q}, Y,Z,\lambda (N-1))\in \pi_{\lambda,Z}\cap\overline{\pi}_{3,\lambda}$ respectively. Here $\overline{\pi}_{\lambda,i}$ denotes the closure of $\pi_{\lambda,i}$ in $\real^4$. 
The respective sets for $\M^-$ are defined in \eqref{pi 1,Lambda}--\eqref{pi 3,4,Lambda} ahead.

\smallskip

With the goal of studying the complementary region to \eqref{region ABQ}, from now on we assume the following hypothesis on the parameters $\alpha, \beta$ from \eqref{def alpha, beta},
\begin{align}\label{A}
0< \alpha, \beta < \tilde{N}_\pm -2,
\end{align}
where the sign $+$ corresponds to the operator $\M^+$, and $-$ to $\M^-$.

\begin{rmk}\label{rmk p or q larger ps}
The hypothesis \eqref{A} ensures that $p$ or $q$ is larger than $\frac{\tilde{N}_\pm}{\tilde{N}_\pm -2}$. In fact, a pair $(p,q)$ with $p,q \le \frac{\tilde{N}_\pm}{\tilde{N}_\pm -2}$ belongs to the region \eqref{region ABQ}.
\end{rmk}

\begin{lem}[$\M^+$]\label{stationary M+}
Under assumption \eqref{A}, the stationary points of the dynamical system \eqref{DS+} in $\K$ are given as follows:
\begin{center}
$O=(0,0,0,0)$,\quad
$N_0=(0,0,\lambda N,\lambda N)$,\quad
$M_0=(X_0,Y_0,Z_0,W_0)$,\quad $A_0=(\tilde{N}_+-2,,\tilde{N}_+-2,0,0)$,
\vspace{0.1cm}
 
$I_0=(\tilde{N}_+-2,0,0,0)$,\quad
$J_0=(0,\tilde{N}_+-2,0,0)$,\quad
$K_0=(0,0,\lambda N, 0)$,\quad
$L_0=(0,0,0,\lambda N)$,
\vspace{0.1cm}

$P_0=(\tilde{N}_+-2, -2+q(\tilde{N}_+-2),0, \Lambda(\tilde{N}_+-q(\tilde{N}_+-2)))$, \;
$G_0=(\tilde{N}_+-2,0,0,\Lambda (\tilde{N}_+ -q(\tilde{N}_+-2))$,
\vspace{0.1cm}

$Q_0=(-2+p(\tilde{N}_+-2),\tilde{N}_+-2,\Lambda(\tilde{N}_+-p(\tilde{N}_+-2)),0)$\; $H_0=(0,\tilde{N}_+-2, \Lambda (\tilde{N}_+ -p(\tilde{N}_+-2),0)$,
\end{center}
where $X_0 = \alpha$, $Y_0=\beta$, $Z_0 =  \Lambda(\tilde{N}_+-2 - \alpha)$, $W_0 =  \Lambda(\tilde{N}_+-2 - \beta)$.
\end{lem}

\begin{proof} We already noticed that $\dot{X}>0$ in $R^+_{\lambda,Z}\cup \pi_{\lambda,Z}$, and $\dot{Y}>0$ in $R^+_{\lambda,W}\cup\pi_{\lambda,W}$. 
In particular, no stationary points are admissible on the concavity hyperplanes $\pi_{\lambda,Z}$ and $\pi_{\lambda,W}$, neither in the interior of the regions $R^+_{\lambda,Z}$ and $R^+_{\lambda,W}$.

On the boundary of $R^+_{\lambda,W}$, $\dot{X}=0$ implies $X=0$. In this case, by \eqref{pi 3,4,lambda}, $\pi_{4,\lambda}\subset R^+_{\lambda,W}$ and $W=\lambda N$, so one gets the point $N_0$.
Meanwhile, from $W=0$ we obtain the point $L_0$.
Analogously, on the boundary of ${R}^+_{\lambda,W}$, $\dot{Y}=0$ yields $Y=0$, and so we derive the point $K_0$. 

Therefore, all the other points are computed in the intersection of the regions $\overline{R}^-_{\lambda,Z}$ and $\overline{R}^-_{\lambda,W}$ with respect to either the coordinate hyperplanes or the planes defined in \eqref{pi 1,lambda}-\eqref{pi 3,4,lambda}. Note that this corresponds to stationary points in the case of the Laplacian operator in dimension $\tilde{N}_+$ for both $u,v$. 
\end{proof}

\begin{rmk}
	The assumption \eqref{A} ensures that $M_0\in \K$. 
\end{rmk}

Analogously, for the operator $\M^-$ we define:
\begin{align}\label{pi 1,Lambda}
\pi_{1,\Lambda} = \{\,(X,Y,Z,W):\, Z = \lambda (\tilde{N}_--2)-\lambda X\,\}\cap \K,
\end{align}
\vspace{-1cm}
\begin{align}\label{pi 2,Lambda}
\pi_{2,\Lambda} = \{\,(X,Y,Z,W):\, W = \lambda (\tilde{N}_--2)-\lambda Y\,\}\cap \K,
\end{align}
\vspace{-1cm}
\begin{align}\label{pi 3,4,Lambda}
\pi_{3,\Lambda} =
{\pi_{3,\Lambda}^+}\; \cup\;{\pi_{3,\lambda}^-}, \quad
\pi_{4,\Lambda} =
{\pi_{4,\Lambda}^+}\; \cup\;{\pi_{4,\Lambda}^-},
\end{align}
which are the sets such that $\dot{X}=0,\, X>0$; $\dot{Y}=0,\, Y>0$; $\dot{Z}=0, \,Z>0 $; and $\dot{W}=0, \,W>0$ respectively, where
\begin{center}
	$\pi_{3,\Lambda}^+= \{(X,Y,Z,W): Z = \Lambda (N-p Y)\}\cap R^+_{\Lambda,Z}$, $\pi_{3,\Lambda}^-= \{(X,Y,Z,W): Z = \lambda (\tilde{N}_-- p Y)\}\cap R^-_{\Lambda,Z}$,
	\vspace{0.2cm}
	
	$\pi_{4,\Lambda}^+= \{(X,Y,Z,W): W= \Lambda (N-q X)\}\cap R^+_{\Lambda,W}$,\,$\pi_{4,\Lambda}^-= \{(X,Y,Z,W): W = \lambda (\tilde{N}_-- qX)\}\cap R^-_{\Lambda,W}$.
\end{center}
Then one finds out the respective stationary points for $\M^-$.
\begin{lem}[$\M^-$]\label{stationary M-}
Under assumption \eqref{A}, the stationary points of the system \eqref{DS-} in $\K$ are
\begin{center}
		$O=(0,0,0,0)$,\quad
		$N_0=(0,0,\Lambda N,\Lambda N)$,\quad
		$M_0=(X_0,Y_0,Z_0,W_0)$,\quad $A_0=(\tilde{N}_--2,,\tilde{N}_--2,0,0)$,
		\vspace{0.1cm}
		
		$I_0=(\tilde{N}_--2,0,0,0)$,\quad
		$J_0=(0,\tilde{N}_--2,0,0)$,\quad
		$K_0=(0,0,\Lambda N, 0)$,\quad
		$L_0=(0,0,0,\Lambda N)$,
		\vspace{0.1cm}
		
		$P_0=(\tilde{N}_--2, -2+q(\tilde{N}_--2),0, \lambda(\tilde{N}_--q(\tilde{N}_--2)))$,\;
		$G_0=(\tilde{N}_--2,0,0,\lambda (\tilde{N}_- -q(\tilde{N}_--2))$,
		\vspace{0.1cm}
		
		$Q_0=(-2+p(\tilde{N}_--2),\tilde{N}_--2,\lambda(\tilde{N}_--p(\tilde{N}_--2)),0)$,\; $H_0=(0,\tilde{N}_--2, \lambda (\tilde{N}_- -p(\tilde{N}_--2),0)$,
\end{center}
	where $X_0 = \alpha$, $Y_0=\beta$, $Z_0 =  \lambda(\tilde{N}_--2 - \alpha)$, $W_0 =  \lambda(\tilde{N}_--2 - \beta)$.
\end{lem}

As in \cite{BV}, the points $C_0=(0,-2,0,\lambda N)$, $D_0=(-2,0,\lambda N, 0)$, $R_0=(0,-2,\lambda (N+2p),\lambda N)$, and $S_0=(-2,0,\lambda N,\lambda (N+2q))$ are also stationary points for both systems \eqref{DS+} and \eqref{DS-}, but they do not play any role in our analysis since they do not belong to $\overline{\K}$.\smallskip

Let us also define the subsets of hyperplanes
\begin{align}\label{L+-}
L^\pm_X=\{(X,Y,Z,W):\, X=\tilde{N}_\pm-2\}\cap \K, \;\; L^\pm_Y=\{(X,Y,Z,W):\, Y=\tilde{N}_\pm-2\}\cap\K.
\end{align}

The next proposition gathers the crucial dynamics at each stationary point in $\K$. In what follows we  denote by $\mathcal{W}_s (P)$ and $\mathcal{W}_u (P)$ the stable (directions entering) and unstable (directions exiting) manifolds at a stationary point $P$, respectively. Set $\dim_s(P):=\dim (\W_s(P))$, $\dim_u(P):=\dim (\W_u(P))$, see \cite{Hale, Wiggins2003}.

\begin{prop}[$\M^\pm$]\label{local study stationary points}
Assume \eqref{A}, then the following properties are verified for the dynamical systems \eqref{DS+} and \eqref{DS-},
	\begin{enumerate}
\item $($Point $M_{0})$  The point $M_0$ is always a saddle point. More precisely:

(i) there are trajectories converging to $M_{0}$ as $t\rightarrow +\infty$ whose corresponding solutions $(u,v)$ of \eqref{P radial m} or \eqref{P radial m-} satisfy $\lim_{r\rightarrow\infty}r^{\alpha}u=c_1>0$ and $\lim_{r\rightarrow\infty}r^{\beta}v=c_2>0$; 

(ii) there exist trajectories approaching $M_{0}$ as $t\rightarrow -\infty$ with respective solutions $(u,v)$ of \eqref{P radial m} or \eqref{P radial m-} verifying $\lim_{r\rightarrow0}r^{\alpha}u=c_3>0$ and $\lim_{r\rightarrow0}r^{\beta}v=c_4>0$;

(iii) the hyperbola $\tH_\pm$ in \eqref{tH} is equal to the set of points $(p,q)$ for which the linearized system at $M_{0}$ has imaginary roots.
		
\item $($Point $N_{0})$ 
A trajectory exits $N_0$ at $ -\infty$ if and only if its corresponding solution $(u,v)$  of \eqref{P radial m} or \eqref{P radial m-} is regular. Also, $\dim_u  (N_0)=2$ and there are infinitely many  trajectories issued from $N_0$.

\item $($Point $A_{0})$  
(i) If $p(\tilde{N}_\pm-2)>\tilde{N}_\pm$
and $q(\tilde{N}_\pm-2)>\tilde{N}_\pm$ then there exist trajectories converging to $A_{0}$ as $t\rightarrow +\infty$.
If $p\neq q$ then $\dim_s(A_0)=2$.

The corresponding solutions $(u,v)$ of \eqref{P radial m} or \eqref{P radial m-} are such that \vspace{-0.1cm}
\begin{center}
	$\lim_{r\to +\infty} r^{\tilde{N}_\pm-2}u=c_1>0$, \quad and \quad $\lim_{r\to +\infty} r^{\tilde{N}_\pm-2}v=c_2>0$.
\end{center}

(ii) If either $p(\tilde{N}_\pm-2)<\tilde{N}_\pm$ or $q(\tilde{N}_\pm-2)<\tilde{N}_\pm$ (they cannot hold simultaneously by assumption \eqref{A}, see Remark \ref{rmk p or q larger ps}), then $\dim_s(A_0)=1$ with either $Z=0$ or $W=0$ when $p\neq q$. In this case there is no trajectory in $\K$ converging to $A_{0}$ when $t\rightarrow +\infty$.

\item $($Point $P_{0})$ 
\textit{(i)} If $2<(\tilde{N}_\pm-2)q<\tilde{N}_\pm$ then $\dim_s(P_0)=2$.
Further,
\begin{center}
	$\dim\left(\W_s(P_0)\cap\{Z=0\}\right)=1$,
\end{center} \vspace{-0.1cm} and there exist trajectories in $\K$ converging to $P_{0}$ when $t\rightarrow +\infty$. The corresponding solutions $(u,v)$ of \eqref{P radial m} or \eqref{P radial m-} satisfy \vspace{-0.2cm}
\begin{center}
	$\lim_{r\rightarrow +\infty} r^{\tilde{N}_\pm-2}u=c_1>0$,  \quad $\lim_{r\rightarrow +\infty} r^{\kappa}v=c_2>0$, \quad $\kappa:=(\tilde{N}_\pm-2)q-2$;  
\end{center}

\textit{(ii)} If $q(\tilde{N}_\pm-2)=\tilde{N}_\pm$ then  $P_0=A_0$ and there is a trajectory in $\K$ converging to this point as $t\to +\infty$. The corresponding solutions $(u,v)$ of \eqref{P radial m} or \eqref{P radial m-} have decay \vspace{-0.1cm}
\begin{align}\label{decay P0}
\textrm{$\lim_{r\to +\infty} r^{\tilde{N}_\pm-2}\,u=c_1>0$, \quad 
	and \quad $\lim_{r\to +\infty} r^{\tilde{N}_\pm-2}\,|\mathrm{ln}r|^{-1}\,v=c_2>0$.}
\end{align}

\item $($Point $Q_{0})$ \textit{(i)} If $2<(\tilde{N}_\pm-2)p<\tilde{N}_\pm$ then  $\dim_s(Q_0)=2$, 
\begin{center}
	$\dim\left(\W_s(Q_0)\cap\{W=0\}\right)=1$,
\end{center} \vspace{-0.1cm}
and there exist trajectories converging to $Q_{0}$ when $t\rightarrow +\infty$. The corresponding solutions $(u,v)$ of \eqref{P radial m} or \eqref{P radial m-} verify \vspace{-0.2cm}
\begin{center}
	$\lim_{r\rightarrow +\infty} r^{\ell}u=c_1>0$, \quad  $\lim_{r\rightarrow +\infty} r^{\tilde{N}_\pm-2}v=c_2>0$,\quad $\ell:=(\tilde{N}_\pm-2)q-2$. \vspace{-0.1cm}
\end{center}
	\end{enumerate}

\textit{(ii)} If  $p(\tilde{N}_\pm-2)=\tilde{N}_\pm$ thus $Q_0=A_0$ and there exists a trajectory in $\K$ converging to $A_{0}$ as $t\rightarrow +\infty$, with corresponding solutions $(u,v)$ of \eqref{P radial m} or \eqref{P radial m-} such that
\vspace{-0.1cm}
\begin{center}
	$\lim_{r\to +\infty} r^{\tilde{N}_\pm-2}\,|\mathrm{ln}r|^{-1}\,u=c_1>0$, \quad and \quad 
	$\lim_{r\to +\infty} r^{\tilde{N}_\pm-2}\,v=c_2>0$.
\end{center} 
Further, there is no trajectory converging to any of the points $O$, $K_{0}$, $L_{0}$, $I_0$, $J_0$, $G_0$, $H_0$ when $r\to +\infty$. 
\end{prop}

\begin{proof}
The dynamics at each stationary point depends upon the linearization of the systems \eqref{DS+} and \eqref{DS-}. 
Since each stationary point belongs to the interior of a region where the concavity of $u,v$ is well defined and they coincide, then our systems correspond to a standard Lane-Emden system involving the Laplacian operator, either in dimension $N$ or in dimension $\tilde{N}_\pm$.
So, the local analysis stated in items 1--5 is implied by \cite[Propositions 4.1--4.11]{BV}, with the exception of item 4(ii) and 5(ii). For the latter, we need to gather some techniques employed in \cite[Section 4.3]{VanderUni} and \cite[Theorem 1.4 (2)]{BV} on the critical case.
We present some details in what follows, for reader's convenience. 
To fix the ideas we consider the operator $\M^+$; for $\M^-$ it is analogous.

Note that the linearization around a stationary point is written as 
\begin{align*}
L(X,Y,Z,W)= \left(
\begin{array}{cccc}
2X - (\tilde{N}-2) + \frac{Z}{\iota} & 0 & \frac{X}{\iota} & 0 \\
0 & 2Y - (\tilde{N}-2) + \frac{W}{\iota} & 0 & \frac{Y}{\iota}  \\
0 & -pZ & \tilde{N}-pY-\frac{2Z}{\iota} & 0\\
-qW & 0 &0 & \tilde{N}-qX-\frac{2W}{\iota}
\end{array}
\right),
\end{align*}
where $\tilde{N}=N$ and $\iota=\lambda$ in $ R^+_{\lambda,Z}\cap R^+_{\lambda,W} $, while $\tilde{N}=\tilde{N}_+$ and $\iota=\Lambda$ in $ R^-_{\lambda,Z}\cap R^-_{\lambda,W} $.

\medskip

\textit{2.} The computation of $L(N_0)$ produces the eigenvalues $2$ and $-N$, each one with eigenspaces of dimension two. Also, the eigenvectors associated with the eigenvalue $2$ have the form 
\begin{align*}
\textstyle	(X,Y,Z,W), \quad \textrm{where } \;\;\; Z= -\frac{p\lambda N}{N+2}Y, \quad W=-\frac{q\lambda N}{N+2}X,
\end{align*} 
that is, the tangent unstable plane is spanned by the eigenvectors
$(1,0, 0,-\frac{q\lambda N}{N+2})$ and 
$(0,1,-\frac{p\lambda N}{N+2},0)$. 
Thus, the tangent plane to the stable manifold at $N_0$ is described by
\begin{align}\label{eigenvector N0}
\textstyle 	(X,Y,\,\lambda N-\frac{p\lambda N}{N+2}Y,\,\lambda N-\frac{q\lambda N}{N+2}X), 
\end{align}
from which the statement of item 2 follows.

\textit{3.} The eigenvectors of the linearization around $A_0$ related to the negative eigenvalues $\lambda_1=\tilde{N}_+-p(\tilde{N}_+-2)$ and $\lambda_2=\tilde{N}_+-q(\tilde{N}_+-2)$ satisfy
\begin{center}
	$X=-\frac{\tilde{N}_+-2}{ \Lambda (p(\tilde{N}_+-2)-2) } Z$, \;  \; $Y=W=0$\; regarding $\lambda_1$;
	
 $Y=-\frac{\tilde{N}_+-2}{\Lambda(q(\tilde{N}_+-2)-2)}W$, \; $X=Z=0$ \; with respect to $\lambda_2$;
\end{center}
whenever $p\neq q$, which gives a plane spanned by the vectors
\begin{center}
	$\left(-\frac{\tilde{N}_+-2}{ \Lambda (p(\tilde{N}_+-2)-2) }  ,0,1,0\right)$ \; and \; $\left(0,-\frac{\tilde{N}_+-2}{\Lambda(q(\tilde{N}_+-2)-2)},0,1\right )$
\end{center}
and translated to the point $A_0$, namely
\begin{align}\label{eigenvector A0}
\textstyle	\left( \tilde{N}_+-2 -\frac{\tilde{N}_+-2}{ \Lambda (p(\tilde{N}_+-2)-2) } Z \, , \, \tilde{N}_+-2-\frac{\tilde{N}_+-2}{\Lambda(q(\tilde{N}_+-2)-2)}W \, ,\,Z,W \right).
\end{align}

\smallskip
	
\textit{4.} 
$L(P_0)$ has negative eigenvalues $\sigma_1=(pq-1)(\alpha - \tilde{N}_+ +2)$ and $\sigma_2=q(\tilde{N}_+ -2)-\tilde{N}_+$, whose eigenvectors satisfy\vspace{-0.1cm}
\begin{align}\label{exp X,Z P0}
\textstyle c_1 X+\frac{\tilde{N}_+-2}{\Lambda}Z=0,\;
c_2Y+\frac{\kappa}{\Lambda}W=0,\; c_3X+c_4W=0 \; \textrm{ for } \sigma_1;
\end{align}\vspace{-0.8cm}
\begin{align}\label{eig P0 sigma2}
\textstyle X=Z=0,\; (\tilde{N}_+-2 )Y+\frac{\kappa}{\Lambda} W=0 \;  \textrm{ for } \sigma_2;
\end{align}
where
$c_1=pq (\tilde{N}_{+}-2)-2(p+1)=pq(\tilde{N}_+ -2-\alpha)+\alpha>0$,  $c_2=\kappa -\sigma_1$, $c_3=q\Lambda\sigma_2$, $c_4=\sigma_2-\sigma_1$, and $\kappa=(\tilde{N}_+-2)q-2$.
Also, $c_3<0$ in the case \textit{(i)}, and $c_3=0$ for \textit{(ii)}.
\smallskip

Next we write $P_0=(X_*,Y_*,0,W_*)$, where $X_*=\tilde{N}_+-2$, $Y_*=-2+q(\tilde{N}_+-2)$, $W_*= \Lambda (\tilde{N}_+-q(\tilde{N}_+-2))$, then
\begin{align*}
L(P_0)= \left(
\begin{array}{cccc}
X_* & 0 & \frac{X_*}{\Lambda} & 0 \\
0 & Y_* & 0 & \frac{Y_*}{\Lambda}  \\
0 & 0 & \tilde{N}_+ -pY_* & 0\\
-qW_* & 0 &0 & -\frac{W_*}{\Lambda}
\end{array}
\right) ,
\end{align*}
whose eigenvalues are given by $X_*>0$, and $Y_*$ which is positive when $q<\frac{\tilde{N}_+}{\tilde{N}_+-2}$, in addition to 
\begin{center}
	$\sigma_1=\tilde{N}_+ -pY_*=(pq-1)(\alpha -\tilde{N}_+ +2)$	 \; and \; $\sigma_2=-\tilde{N}_+ +q(\tilde{N}_+-2)=-\frac{W_*}{\Lambda}$
\end{center} which are negative. 
Writing $X=X_*+x$, $Y=Y_*+y$, $Z=z$, and $W=W_*+w$, one has 
\begin{align}\label{L(P0)}
\left(
\begin{array}{cccc}
\dot{x} \\
\dot{y} \\
\dot{z}\\
\dot{w}
\end{array}
\right)  
=
\left[
\begin{array}{cccc}
(X_*+x)\,(x+\frac{z}{\Lambda}) \\
(Y_*+y)\, (y+\sigma_2+\frac{w}{\Lambda})  \\
\sigma_1 z -z(py +\frac{z}{\Lambda})\\
-(W_*+w) \,(qx+\sigma_2+\frac{w}{\Lambda})
\end{array}
\right] 
\approx 
L(P_0) \left(
\begin{array}{cccc}
x \\
y \\
z\\
w
\end{array}
\right)  
=
\left[
\begin{array}{cccc}
X_* \,(x+\frac{z}{\Lambda}) \\
Y_*\, (y+\frac{w}{\Lambda})  \\
\sigma_1\, z\\
\sigma_2 \, (q\Lambda x+{w})
\end{array}
\right] .
\end{align}

We first observe that the intersection of the stable manifold at $P_0$ with the plane $Z=0$ has dimension one  because of \eqref{eig P0 sigma2}. Now, by the third equation in \eqref{L(P0)}, there exists a trajectory with $Z>0$ converging to $P_{0}$ when $t\rightarrow +\infty$. Since $\sigma_1<0$, the convergence of $Z$ to $0$ is exponential. Then, by \eqref{exp X,Z P0} we see that also $X$ converges to $\tilde{N}_+-2$ exponentially.
In particular, the decay of $u$ in \eqref{decay P0} holds, see \cite[proof of equation (3.7)]{MNPscalar}.

\smallskip

\textit{4(i)}.
Let us prove that the tangent stable plane at $P_0$ is spanned by  eigenvectors in the form
$(0,C_{1}, 0,1)$ and $(C_2, C_3,1,C_4)$, 	
for some nonzero constants $C_i$ for $i=1,\ldots, 4$, depending only on $N,\lambda,\Lambda,p,q$.
In particular, $X=X(Z)$ and $Y=Y(Z,W)$.

Considering the stable direction associated with $\sigma_1$ by \eqref{exp X,Z P0} with $c_3<0$, we first notice that if $c_2=0$, then independently of the sign of $c_4$ we have $X=Z=W$, and so the main direction of $\sigma_1$ is given by $(0,Y,0,0)$.
On the other hand, if $c_4=0$ and $c_2\neq 0$ then we get $X=Z=0$ and so $(0,Y(W),0,W)$.
However, these two cases are not admissible, otherwise there would not exist a trajectory with $Z>0$ arriving at $P_0$. 

Therefore, we need to have both $c_2\ne 0$ and $c_4\ne 0$, which means that the main direction at $\sigma_1$ reads as $(X(W),Y(W), Z(W),W)$ by \eqref{exp X,Z P0}.
Observe that $c_4=0 $ if and only if $c_2=\kappa - \sigma_2=\tilde{N}_+ -2$, in which case the two main stable directions given by $\sigma_1$ and $\sigma_2$ coincide. Hence $c_2\neq \tilde{N}_+ -2$ and these directions are linearly independent, producing a two dimensional stable manifold.

We may rewrite the stable direction at $\sigma_1$ as $(X(Z),Y(Z), Z,W(Z))$.
Also, the main stable direction at $\sigma_2$ is described through $(0,Y(W),0,W)$.
In particular, $\W_s(P_0)$ is a two dimensional graph over the variables $(Z,W)$. 
Analogously one concludes item 5 for $Q_0$, by exchanging the roles of $p$ and $q$, $X$ and $Y$, $Z$ and $W$, and $\alpha$ and $\beta$.

\smallskip

\textit{4(ii).} We assume $q=\frac{\tilde N_+}{\tilde{N}_+-2}$ and $p>\frac{\tilde{N}_+}{\tilde{N}_+-2}$. The eigenvalues at $A_0=P_0$ are $X_*=\tilde{N}_+ -2$ with multiplicity 2, in addition to $\sigma_1=\tilde{N}_+ - p(\tilde{N}_+ -2)<0$, and $\sigma_2=\tilde{N}_+ - q(\tilde{N}_+ -2)=0$. Recall that $c_3=0$ in \eqref{exp X,Z P0} in this case.
\smallskip

Now, the linearization around the zero eigenvalue does not provide any information. So we address it by a center manifold argument, by showing its existence jointly with the stable and unstable manifolds at $A_0=P_0$, in the spirit of \cite[Theorem 3.2.1]{Wiggins2003}. This extends the asymptotic analysis performed in \cite{BV, VanderUni} for the critical regime.

To see this we use \eqref{L(P0)} with $L_0=L(A_0)=L(P_0)$, $X_*=Y_*$, and $\sigma_2=W_*=0$ to get
\begin{align*}
\left(
\begin{array}{cccc}
\dot{x} \\
\dot{y} \\
\dot{z}\\
\dot{w}
\end{array}
\right)  
=
L_0 \left(
\begin{array}{cccc}
x \\
y \\
z\\
w
\end{array}
\right)  
+
 \left[
\begin{array}{cccc}
x(x+\frac{z}{\Lambda}) \\
y (y+\frac{w}{\Lambda})  \\
-z(py+\frac{z}{\Lambda})\\
-w(qx+\frac{w}{\Lambda})
\end{array}
\right]
,\quad
L_0= \left(
\begin{array}{cccc}
X_* & 0 & \frac{X_*}{\Lambda} & 0 \\
0 & X_* & 0 & \frac{X_*}{\Lambda}  \\
0 & 0 & \sigma_1 & 0\\
0 & 0 &0 & 0
\end{array}
\right) .
\end{align*}
Next, as in \cite{GuerraJDE12}, we write $P^{-1}L_0 \,  P=D$, where  
\begin{align*}
D= \left(
\begin{array}{cccc}
X_* & 0 &0 & 0 \\
0 & X_* & 0 & 0  \\
0 & 0 & \sigma_1 & 0\\
0 & 0 &0 & 0
\end{array}
\right) 
, \quad P=  \left(
\begin{array}{cccc}
1 & 0 & -c & \,0 \\
0 & 1 & 0 &-\frac{1}{\Lambda} \\
0 & 0 & 1 & \,0\\
0 & 0 &0 & \,1
\end{array}
\right), \quad \textstyle c=\frac{X_*}{\Lambda(X_*-\sigma_1)}.
\end{align*}
Here, the columns of $P$ form a basis of eigenvectors associated to the eigenvalues $X_*$, $\sigma_1$ and $0$, respectively. Then we introduce the new variables $\bar x, \bar y, \bar z, \bar w$ satisfying
\begin{align*}
\left(
\begin{array}{cccc}
x \\
y \\
z\\
w
\end{array}
\right)
=P
\left(
\begin{array}{cccc}
\bar x \\
\bar y \\
\bar z\\
\bar w
\end{array}
\right)
=
\left(
\begin{array}{cccc}
\bar x- c\bar z\\
\bar y-\frac{\bar w}{\Lambda} \\
\bar z\\
\bar w
\end{array}
\right),
\end{align*} 
from which $\bar z=z$, $\bar w=w$, and so $\bar x =x+cz$,  $\bar y= y+\frac{w}{\Lambda}$. Since $c(\sigma_1-X_*)+\frac{X_*}{\Lambda}=0$, it holds
\begin{align*}
\begin{cases}
\;\dot{\bar x}=
\dot x  +c \dot z 
=(X_*+x)(x+\frac{z}{\Lambda})+cz (\sigma_1-py-\frac{z}{\Lambda}) 
=X_*(x+cz)+x(x+\frac{z}{\Lambda})-cz(py+\frac{z}{\Lambda})
, 
\\
\;\dot{\bar y}=
\dot y +\frac{\dot w}{\Lambda}
=X_*(y+\frac{w}{\Lambda})+(y+\frac{w}{\Lambda})-\frac{w}{\Lambda}(qx+\frac{w}{\Lambda}).
\end{cases}
\end{align*}
Hence $(\bar x, \bar y, \bar z, \bar w)$ solves the following system 
\begin{align*}
\left(
\begin{array}{cccc}
\dot{\bar x} \\
\dot{\bar y} \\
\dot{\bar z}\\
\dot{\bar w}
\end{array}
\right)  
=
D \left(
\begin{array}{cccc}
\bar x \\
\bar y \\
\bar z\\
\bar w
\end{array}
\right)  
+ \frak f \left(
\begin{array}{cccc}
\bar x \\
\bar y \\
\bar z\\
\bar w
\end{array}
\right)  , \;\;\; 
\frak f \left(
\begin{array}{cccc}
\bar x \\
\bar y \\
\bar z\\
\bar w
\end{array}
\right)  
=
\left[
\begin{array}{cccc}
(\bar x  -c \bar z)(\bar x-c\bar z +\frac{\bar z}{\Lambda})-c\bar z(p\bar y -\frac{p\bar w}{\Lambda}+\frac{\bar z}{\Lambda}) \\
(\bar y-\frac{\bar w}{\Lambda})\bar y  -\frac{\bar w}{\Lambda}(q\bar x -qc\bar z +\frac{\bar w}{\Lambda}) \\
-\bar z(p\bar y-\frac{p\bar w}{\Lambda}+\frac{\bar z}{\Lambda})\\
-\bar w(q\bar x - q c\bar z+\frac{\bar w}{\Lambda})
\end{array}
\right],
\end{align*}
with $\frak f (0)=0$ and $D\frak f (0)=0$.
Now,  by using \cite[Theorem 3.2.1]{Wiggins2003}\footnote{See also Theorem 2.1.1 in the third corrected printing of this book, 1996.}, we obtain the existence of a unique center manifold around the point $A_0=P_0$, described near $\bar w =0$ as
\begin{center}
$\{   
\,(\bar x, \bar y, \bar z, \bar w): \,\bar x=h_1(\bar w), \;\bar y=h_2(\bar w), \;\bar z=h_3(\bar w),\; h_i(0)=h_i^\prime (0)=0,\, i=1,2,3\,
\}$.
\end{center}

Here we may use Taylor's expansions for $h_i$ as
\begin{center}
$h_i(w)=a_i w^2+O(w^3)$ \; as $w\to 0$, \; $i=1,2,3$ \; ($w=\bar w$).
\end{center}
Next, we have locally,
\begin{center}
$\dot w = - w\,\{ q h_1(w) - q c\,h_3(w)+\frac{ w}{\Lambda}\} =-\frac{w^2}{\Lambda}+O(w^3)$,
\end{center}
from which we deduce that $w(t)=\frac{\Lambda}{ t}+O(t^{-\gamma})$, for some $\gamma>1$. Thus, 
\begin{center}
	$y(t)=h_2(w)-\frac{w}{\Lambda}=-\frac{1}{t}+O(t^{-\gamma})$\, as $t\to +\infty$, \;  $\gamma>1$.
\end{center}

By the theory of center manifolds, a trajectory $\tau$ enters $A_0=P_0$, via the center manifold just found. This gives the behavior of $Y=\tilde N_+-2 +y$ as $t\to +\infty$, and in turn the decay \eqref{decay P0} of the corresponding function $v$ as $r\to \infty$.

Similarly, the analysis at $Q_0$ is performed by
writing $Q_0=(X_*,Y_*,Z_*,0)$, where  $X_*=-2+p(\tilde{N}_+-2)$, $Y_*=\tilde{N}_+-2$, $Z_*= \Lambda (\tilde{N}_+-p(\tilde{N}_+-2))$, in addition to
 \begin{center}
	$\sigma_1=\tilde{N}_+ -qX_*=(pq-1)(\beta -\tilde{N}_+ +2)$	 \; and \; $\sigma_2=-\tilde{N}_+ +p(\tilde{N}_+-2)=-\frac{Z_*}{\Lambda}$
\end{center} which are negative,  
$X=X_*+x$, $Y=Y_*+y$, $Z=Z_*+z$, $W=w$,  and
\begin{align}\label{L(Q0)}
\left(
\begin{array}{cccc}
\dot{x} \\
\dot{y} \\
\dot{z}\\
\dot{w}
\end{array}
\right)  
=
\left[
\begin{array}{cccc}
(X_*+x)\,(x+\sigma_2+\frac{z}{\Lambda}) \\
(Y_*+y)\, (y+\frac{w}{\Lambda})  \\
-(Z_*+z) \,(py+\sigma_2+\frac{z}{\Lambda})\\
\sigma_1 w -w(qx +\frac{w}{\Lambda})
\end{array}
\right] 
\approx 
L(Q_0) \left(
\begin{array}{cccc}
x \\
y \\
z\\
w
\end{array}
\right)  
=
\left[
\begin{array}{cccc}
X_* \,(x+\frac{z}{\Lambda}) \\
Y_*\, (y+\frac{w}{\Lambda})  \\
\sigma_2 \, (p\Lambda y+{z})\\
\sigma_1\, w
\end{array}
\right] .
\end{align}
\end{proof}

Next we analyze the directions of the vector field $F$ in \eqref{Prob x=(X,Z)} on the $X,Z$ axes, on the concavity sets $\pi_{\lambda,Z}$, $\pi_{\lambda,W}$, $\pi_{\Lambda,Z}$, $\pi_{\Lambda,W}$ as well as on $\pi_{1,\lambda}$ -- $\pi_{4,\lambda}$ and $\pi_{1,\Lambda}$ -- $\pi_{4,\Lambda}$.

\begin{prop}\label{prop flow}
	Given a trajectory $\tau=(X,Y,Z,W)$ of the system \eqref{DS+}, the following properties are verified for the operator $\M^+$ in the region $\K$.\vspace{-0.1cm}
	\begin{enumerate}[(i)]
		\item If $\tau$ crosses $\pi_{\lambda,Z}$, then $Y>1/p$ when passing from $R^+_{\lambda,Z}$ to $R^-_{\lambda,Z}$, and $Y<1/p$ when passing from $R^-_{\lambda,Z}$ to $R^+_{\lambda,Z}$. If instead $\tau$ crosses $\pi_{\lambda,W}$, then $X>1/q$ when passing from $R^+_{\lambda,W}$ to $R^-_{\lambda,W}$, and $X<1/q$ when passing from $R^-_{\lambda,W}$ to $R^+_{\lambda,W}$.\vspace{-0.2cm}
		
\item If $X(T)=\tilde{N}_+-2$ then $\dot{X}>0$ from $T$ on; in particular, if $\tau$ crosses $L^+_X$ then it never turns back by crossing $L_X^+$ (see \eqref{L+-}) another time. Also, $\dot{W}<0$ from $T$ on if $q\geq \frac{\tilde{N}_+}{\tilde{N}_+-2}$.
		\vspace{-0.2cm}
		
		Analogously, if $\tau$ crosses $L^+_Y$ (see \eqref{L+-}) at time $T$ thus $\dot{Y}>0$ from $T$ on.	Further, we have $\dot{Z}<0$ from $T$ on if $p\geq \frac{\tilde{N}_+}{\tilde{N}_+-2}$.
		\vspace{-0.2cm}
		
		\item If $Z(T)=\lambda N$ then $\dot{Z}<0$ in $\K$ for all $t\le T$. 
		Similarly, if $W(T)=\lambda N$ then $\dot{W}<0$ in $\K$ for all $t\le  T$. In particular, a regular trajectory satisfies $Z,W< \lambda N$ whenever it is defined.
		\vspace{-0.2cm}
		
\item The set $\pi_{1,\lambda}$ in \eqref{pi 1,lambda} lies on the left hand side of $L_X^+$; while $\pi_{2,\lambda}$ in \eqref{pi 2,lambda} is on the left hand side of $L_Y^+$. Moreover, a point $P=(X,Y,Z,W)$ on $\pi_{3,\lambda}$ or $\pi_{4,\lambda}$ in \eqref{pi 3,4,lambda} satisfies:
	\vspace{-0.2cm}
\begin{center}
$Y< 1/p$\;\; if $P\in \pi_{3,\lambda}^+$ , \quad $Y<\tilde{N}_+/p$\;\; if $P\in\pi_{3,\lambda}^-$ , 

$X< 1/q$\;\; if $P\in \pi_{4,\lambda}^+$ , \quad $X<\tilde{N}_+/q$\;\; if $P\in\pi_{4,\lambda}^-$ .
\end{center}
	\vspace{-0.3cm}
Here, $\pi_{3,\lambda}^-$ is on the left hand side of $L_Y^+$ when $p\geq \frac{\tilde{N}_+}{\tilde{N}_+-2}$, while $\pi_{4,\lambda}^-$ is on the left hand side of $L_X^+$ if $q\geq \frac{\tilde{N}_+}{\tilde{N}_+-2}$.	
		\vspace{-0.2cm}
	\end{enumerate}
	The same results are true for the system \eqref{DS-} when the operator is $\M^-$, by exchanging the roles of $\lambda$ and $\Lambda$, $X$ and $Y$, $p$ and $q$, $Z$ and $W$, and replacing $\tilde{N}_+$ by $\tilde{N}_-$.
\end{prop}

\begin{proof} Let us consider the operator $\M^+$, since the proof for $\M^-$ is analogous.
	
	$(i)$ One has $\dot{Z}=Z(1-pY)$ on $\pi_{\lambda,Z}$, and $\dot{W}=Z(1-qX)$ on $\pi_{\lambda,W}$.
	
	\vspace{0.15cm}
	
	$(ii)$ If $\tau(T)\in R_{\lambda,Z}^-$, then $\dot{X}> X(X-(\tilde{N}_+-2))=0$ at time $T$, since $X,Z>0$ in $\K$. On the other hand, if $\tau(T)\in R_{\lambda,Z}^+\cup\pi_{\lambda,Z}$ then $\dot{X}(T)>0$ by \eqref{pi 1,lambda}. 
	This shows that the vector field on $L^+_X$ is going out, so the trajectory $\tau$ can never turn back in the $X$ direction after intersecting it. 

Next, for $q\geq \frac{\tilde{N}_+}{\tilde{N}_+-2}$, we have $\dot{W}\le W(\tilde{N}_+ -q(\tilde{N}_+-2)-\frac{W}{\Lambda})<0$ for $t\geq T$ in the region $R_{\lambda,W}^-$; while $\dot{W}\le W(1 -q(\tilde{N}_+-2))<0$ for $t\geq T$ in $R_{\lambda,W}^+\cup \pi_{\lambda,W}$, since $W\geq \lambda(N-1)$ and $\tilde{N}_+>1$.	
	\vspace{0.2cm}
	
	$(iii)$ The hyperplane subset $\chi=\{(X,Y,Z,W)\in \K:Z= \lambda N\}$ and the region in $\K$ above $\chi$ are contained in $R^+_{\lambda,Z}$. Moreover,\vspace{-0.04cm}
\begin{center}
 $\dot{Z}=Z(N-{Z}/{\lambda}-pY)< 0$\; in $\K$\;\; if \;$Z\ge \lambda N$,\vspace{-0.02cm}
\end{center}
since $Y,Z> 0$. In particular, if $\tau$ crosses $\chi$ at the time $T$, then $Z>\lambda N$ for all $t< T$.

On the other hand, if $\tau$ is a regular trajectory, then it starts at $N_0=(0,0,\lambda N, \lambda N)$ at $ -\infty$, by Proposition \ref{local study stationary points} (2). But it can never reach $\chi$ since the vector field on $\chi$ is pointing down.

	\vspace{0.15cm}
$(iv)$ This comes from the definition of the hyperplanes in \eqref{pi 1,lambda}--\eqref{pi 3,4,lambda} replaced into the equations of the system \eqref{DS+}. In fact, one has
\begin{center}
$X=\tilde{N}_+-2-\frac{Z}{\Lambda}<\tilde{N}_+-2$ \; on $\pi_{1,\lambda}$ ,\quad $Y=\frac{\tilde{N}_+}{p}-\frac{Z}{\Lambda p}<\frac{\tilde{N}_+}{p}$ \; on $\pi_{3,\lambda}^-$ ,
\end{center}  
and further
\begin{center}
$Y=\frac{N}{p}-\frac{Z}{\lambda p}< \frac{1}{p}$ \; on\, ${\pi_{3,\lambda}^+}$ \;\; since \;$Z> \lambda (N-1)$ \; in $R^+_{\lambda,Z}$.
\end{center}
Analogously one verifies the statements for $\pi_{2,\lambda}$ and $\pi_{4,\lambda}$.
We observe that $\tilde{N}_+/p\leq \tilde{N}_+-2$  if $p\geq \frac{\tilde{N}_+}{\tilde{N}_+-2}$, while $\tilde{N}_+/q\leq \tilde{N}_+-2$ if $q\geq \frac{\tilde{N}_+}{\tilde{N}_+-2}$.
\end{proof}

\section{Global study}\label{section global}

\subsection{A priori bounds}

We start by deriving some a priori bounds for trajectories of the systems \eqref{DS+} and \eqref{DS-} which are defined backward or forward for all time.

\begin{prop}\label{AP bounds}
	Let $\tau$ be a trajectory of \eqref{DS+} or \eqref{DS-} in $\K$, $\tau(t)=(X(t),Y(t),Z(t),W(t))$.
	\begin{enumerate}[(i)]
		\item  If $\tau$ is defined in $[\hat{t},+\infty)$ for some $\hat{t}\in \real$, then $X(t),Y(t)< \tilde{N}_\pm-2$ for all $t\geq \hat{t}$. 
		
		\item  If instead $\tau$ is defined in $(-\infty,\hat{t}]$ for some $\hat{t}\in \real$, then $Z(t),W(t)<\lambda N$ in the case of $\M^+$, or $Z(t),W(t)<\Lambda N$ for $\M^-$, for all $t\leq \hat{t}$.
	\end{enumerate}
	
	In particular, if $\tau$ is a global trajectory of \eqref{DS+} or \eqref{DS-} in $\K$ defined for all $t\in \real$, then $\tau$ remains inside the box
	$\mathcal{B}_+:=(0,\tilde{N}_+-2)\times (0,\tilde{N}_+-2)\times (0,\lambda N)\times (0,\lambda N)$ in the case of $\M^+$; whereas it stays in $\mathcal{B}_-:=(0,\tilde{N}_--2)\times (0,\tilde{N}_--2)\times(0,\Lambda N)\times (0,\Lambda N)$ for $\M^-$.
\end{prop}

\begin{proof}
	We only consider the operator $\M^+$, since the proof for $\M^-$ is simpler by using $\tilde{N}_-\ge N$.
	
	$(i)$ Arguing by contradiction we assume that $X(t_1)\geq \tilde{N}_+-2$ for some $t_1\geq \hat{t}$. Then Proposition \ref{prop flow} (ii) yields $\dot{X}>0$ for all $t\ge  t_1$.
	Therefore, $X$ has a limit when $t\to+\infty$, which is either $+\infty$ or a positive constant $A\ge \tilde{N}_+-2$.
	
	Suppose first $X(t)\rightarrow+\infty$ as $t\to +\infty$. Thus we can choose a time $t_2$ such that $X(t_2)>N-2\ge \tilde{N}_+-2$ for all $t\geq t_2$. In such a scenario the proof reduces to the one for the Laplacian given in \cite{BV}.
	Indeed, the first equation in \eqref{DS+} and $\tilde{N}_+\leq N$ yield
	\begin{align}\label{int1F X}
	\textstyle{ \frac{\dot{X}}{X[X-({N}-2)]}\geq 1 \Rightarrow\; \frac{({N}-2) \dot{X} }{X[X-({N}-2)]}
		=\frac{\dot{X}}{X-({N}-2)}-\frac{\dot{X}}{X}
		=\frac{\rmd}{\rmd t}\,\mathrm{ln} \left(  \frac{X(t)-{N}+2}{X(t)}  \right) \; \textrm{ for all } t\geq t_2.}
	\end{align}
	Thus, by integrating \eqref{int1F X} in the interval $[t_2,t]$ we get
	\begin{align}\label{int2F X}
	\textstyle{X(t)\geq \frac{{N}-2}{1-c e^{({N}-2)(t-t_2)} }}, \;\; \textrm{ where }\;\;\textstyle{c=1-\frac{{N}-2}{X(t_2)}\in (0,1)}.
	\end{align}
	
	In the second case, i.e.\ $X(t)\rightarrow A$ as $t\to +\infty$ with $A\in (0,+\infty)$, we have $\lim_{t\to +\infty}\dot{X}(t)=0$, and so $\tau$ approaches the hyperplane $\pi_{1,\lambda}$. 
	Now, since $\pi_{1,\lambda}$ strictly lies in the region ${R}^-_{\lambda,Z}$, then there is some $T\ge t_1$ such that $\tau\in R^-_{\lambda,Z}$ for all $t\ge T$.
	Then one performs the calculations \eqref{int1F X}, \eqref{int2F X} with $\tilde{N}_+$ in place of $N$, which are similar to those in the scalar case \cite{MNPscalar}, and gets that $X$ blows up in finite time, a contradiction. The proof of $Y<\tilde{N}_+-2$ is analogous.
	
	\smallskip
	
	$(ii)$ By Proposition \ref{prop flow} (iii), if $Z(T)=\lambda N$, then $\tau$ remains in the region $R^+_{\lambda,Z}$ up to $T$.
	Since in $R^+_{\lambda,Z}$ the operator is the Laplacian, the proof in \cite{BV} applies. In fact, it is enough to integrate in $[t,t_0]$ as before to see that $Z$ blows up in finite backward time.
The case $W<\lambda N$ is similar.
\end{proof}

\begin{corol}[Decay of regular solutions]\label{cor decay u,v}
	Every positive solution pair $u,v$  of the problem \eqref{P radial m} or \eqref{P radial m-} defined in $[0,+\infty)$ is bounded and $u(r)\leq Cr^{-\alpha}$ and $ v(r) \leq Cr^{-\beta}$ for all $r>0$.
\end{corol}

\begin{proof}
	This comes from \eqref{def u via X,Z} and the a priori bounds from Proposition \ref{AP bounds}.
\end{proof}

\begin{rmk}\label{Remark 2F}
	Let us observe that if $(u,v)$ is a regular solution of \eqref{LE}, \eqref{H Dirichlet} in the ball $B_R$, $R>0$, then for the corresponding trajectory $\Gamma$ of \eqref{DS+} or \eqref{DS-} both $X$ and $Y$ blow up at the time $T=\mathrm{ln}(R)$. Viceversa, if $\Gamma$ is a trajectory issued from $N_0$ such that $X$ and $Y$ blow up at the same time $T$ then for the corresponding solution pair $(u,v)$ of \eqref{P radial m} or \eqref{P radial m-} it holds $u(R)=v(R)=0$ for $R=e^T$. Thus there exists a radial solution of \eqref{LE}, \eqref{H Dirichlet} in $B_R$. This can be proved as in the scalar case, see \cite[Section 3]{MNPscalar}.
\end{rmk}

\subsection{The complementary cones}

In this section we give a brief analysis on the other cones apart from $\K$. We set:
\begin{center}
	$\K_0=\{\,(X,Y,Z,W)\in \real^4:\; X,Y,Z,W<0\,\}$, \smallskip	
	
	$\mathcal{K}_1=\{\,(X,Y,Z,W)\in \real^4:\; X,Z>0,\; Y,W<0\,\}$,\smallskip
	
	$\mathcal{K}_2=\{\,(X,Y,Z,W)\in \real^4:\; X,Z<0,\; Y,W>0\,\}$.  	
\end{center}
That is, $\K_0$ is the region in $\real^4$ such that the corresponding $(u,v)$ satisfy $u^\prime>0$ and $v^\prime>0$; while $\K_1$ concerns $u^\prime<0$ and $v^\prime>0$; and finally $\K_2$ to those with $u^\prime>0$ and $v^\prime<0$.

By Remark \ref{Rmk sing reg} trajectories corresponding to regular and singular solutions do not intersect $\K_0\cup\K_1\cup \K_2$. However, they will be essential for studying exterior domain solutions.

Since we are considering positive solutions of \eqref{LE}, and $u^\prime>0$ implies $u^{\prime\prime}<0$, as well as $v^\prime>0$ produces $v^{\prime\prime}<0$, one finds out the following systems of ODEs:
\begin{align}\label{P M+ 3Q}
\textrm{\;\;for $\M^+$ in $\K_0$} :\quad \left\{\begin{array}{l}
\lambda u^{\prime\prime}\;=\; -\Lambda r^{-1}(N -1)u^\prime-  v^p , \\
\lambda v^{\prime\prime}\;=\; -\Lambda r^{-1}(N -1)v^\prime-  u^q , \quad u,\, v> 0  ;\end{array}\right.
\end{align}\vspace{-0.5cm}
\begin{align}\label{P M- 3Q}
\textrm{\;\;for $\M^-$ in $\K_0$} :\quad \left\{\begin{array}{l}
\Lambda u^{\prime\prime}\;=\; -\lambda r^{-1}(N -1)u^\prime-  v^p , 
\\ \Lambda v^{\prime\prime}\;=\; -\lambda r^{-1}(N -1)v^\prime-  u^q , \quad u,\, v> 0; \end{array}\right.
\end{align}
\vspace{-0.35cm}
\begin{align}\label{P M+ K1}
\textrm{for $\M^+$ in $\K_1$} :\quad
\left\{
\begin{array}{l}
u^{\prime\prime}\;=\; M_+(-\lambda r^{-1}(N-1)\, u^\prime- v^p ), \\
\lambda v^{\prime\prime}\;=\; -\Lambda r^{-1}(N -1)v^\prime-  u^q ;
\end{array}
\right.
\end{align}
\vspace{-0.5cm}
\begin{align}\label{P M- K1}
\textrm{\;\;for $\M^-$ in $\K_1$} :\quad \left\{
\begin{array}{l}
u^{\prime\prime}\;=\; M_-(-\Lambda r^{-1}(N-1)\, u^\prime- v^p ), \\
\Lambda v^{\prime\prime}\;=\; -\lambda r^{-1}(N -1)v^\prime-  u^q , \quad u,\, v> 0;
\end{array}
\right.
\end{align}
\vspace{-0.35cm}
\begin{align}\label{P M+ K2}
\textrm{for $\M^+$ in $\K_2$} :\quad
\left\{
\begin{array}{l}
\lambda u^{\prime\prime}\;=\; -\Lambda r^{-1}(N -1)u^\prime-  v^p, \\
v^{\prime\prime}\;=\; M_+(-\lambda r^{-1}(N-1)\,v^\prime- u^q ) , \quad u,\, v> 0  ;
\end{array}
\right.
\end{align}
\vspace{-0.5cm}
\begin{align}\label{P M- K2}
\textrm{\;\;for $\M^-$ in $\K_2$} :\quad \left\{
\begin{array}{l}
\Lambda u^{\prime\prime}\;=\; -\lambda r^{-1}(N -1)u^\prime-  v^p , \\
v^{\prime\prime}\;=\; M_-(-\Lambda r^{-1}(N-1)\,v^\prime- u^q ) , \quad u,\, v> 0  .
\end{array}
\right.
\end{align}

Now, in terms of the associated dynamical systems, we get:
\begin{align}\label{DS+ 3Q}\textrm{$\M^+$ in $\K_0$\,:\;\;}
\left\{\begin{array}{cc}
\dot{X} = \;X \, [\,X-(\tilde{N}_- -2)+\frac{Z}{\lambda} \,] , \;& \dot{Y} = \;Y \, [\,Y-(\tilde{N}_- -2)+\frac{W}{\lambda} \,] , \\ \dot{Z}= \;Z\, [\, \tilde{N}_- -pX - \frac{Z}{\lambda}\,], \; &\dot{W}= \;W\, [\, \tilde{N}_- -qX - \frac{W}{\lambda}\,],
\end{array}\right.
\end{align}
\vspace{-0.2cm}
\begin{align}\label{DS+ K1}\textrm{$\M^+$ in $\K_1$\,:\;\;}
\left\{\begin{array}{cc}
\dot{X} =\;X \, [\,X+1-M_+ (\lambda(N-1)-Z )\,], \;& \dot{Y} = \;Y \, [\,Y-(\tilde{N}_- -2)+\frac{W}{\lambda} \,] , \\ 
\dot{Z} = \;Z\, [\,1-pY +M_+ (\lambda(N-1)-Z )\,], \; &\dot{W}= \;W\, [\, \tilde{N}_- -qX - \frac{W}{\lambda}\,],
\end{array}\right.
\end{align}
\vspace{-0.2cm}
\begin{align}\label{DS+ K2}\textrm{$\M^+$ in $\K_2$\,:\;\;}
\left\{\begin{array}{cc}
\dot{X} = \;X \, [\,X-(\tilde{N}_- -2)+\frac{Z}{\lambda} \,] , & 
\dot{Y} = \;Y\, [\,Y+1-M_+ (\lambda(N-1)-W )\,] , \\ 
\dot{Z}= \;Z\, [\, \tilde{N}_- -pX - \frac{Z}{\lambda}\,],  &
\dot{W} = W\, [\,1-qX +M_+ (\lambda(N-1)-W )\,].
\end{array}\right.
\end{align}
Likewise, for the operator $\M^-$ one has:
\begin{align}\label{DS- 3Q}\textrm{$\M^-$ in $\K_0$\,:\;\;}
\left\{\begin{array}{cc}
\dot{X} = \;X \, [\,X-(\tilde{N}_+ -2)+\frac{Z}{\Lambda} \,] , \;
& \dot{Y} = \;Y \, [\,Y-(\tilde{N}_+ -2)+\frac{W}{\Lambda} \,] , \\ 
\dot{Z}= \;Z\, [\, \tilde{N}_+ -pX - \frac{Z}{\Lambda}\,],\; 
&\dot{W}= \;W\, [\, \tilde{N}_+ -qX - \frac{W}{\Lambda}\,],
\end{array}\right.
\end{align}
\vspace{-0.2cm}
\begin{align}\label{DS- K1}\textrm{$\M^-$ in $\K_1$\,:\;\;}
\left\{\begin{array}{cc}
\dot{X} =\;X \, [\,X+1-M_- (\Lambda(N-1)-Z )\,], \;& \dot{Y} = \;Y \, [\,Y-(\tilde{N}_+ -2)+\frac{W}{\Lambda} \,] , \\ 
\dot{Z} = \;Z\, [\,1-pY +M_- (\Lambda(N-1)-Z )\,], \; &\dot{W}= \;W\, [\, \tilde{N}_+ -qX - \frac{W}{\Lambda}\,],
\end{array}\right.
\end{align}
\vspace{-0.2cm}
\begin{align}\label{DS- K2}\textrm{$\M^-$ in $\K_2$\,:\;\;}
\left\{\begin{array}{cc}
\dot{X} = \;X \, [\,X-(\tilde{N}_+ -2)+\frac{Z}{\Lambda} \,] , & 
\dot{Y} = \;Y\, [\,Y+1-M_- (\Lambda(N-1)-W )\,] , \\ 
\dot{Z}= \;Z\, [\, \tilde{N}_+ -pX - \frac{Z}{\Lambda}\,],  &
\dot{W} = W\, [\,1-qX +M_- (\Lambda(N-1)-W )\,].
\end{array}\right.
\end{align}
\begin{rmk}
No stationary points exist in $\K_0\cup\K_1\cup\K_2$. Indeed, since $v^\prime >0$ and $v>0$ yields $v^{\prime\prime}>0$ then $\dot{Y}>0$ in $\K_1$. Analogously, $\dot{X}>0$ in $\K_2$, and $\dot{X}, \dot{Y}>0$ in $\K_0$. 
In other words, we do not have stationary points out of $\K$ when we are considering positive solutions $(u,v)$ of \eqref{LE}.
\end{rmk}

\subsection{Uniqueness of regular solutions}

We first consider the case of the $\M^+$ system.
Let $(u,v)$ be a regular solution of the initial value problem \eqref{shooting} and $\Gamma$ its corresponding trajectory of \eqref{DS+} in $\K$.
The relations
\begin{center}
	$XZ=r^2\frac{v^p}{u}$, \qquad $YW=r^2\frac{u^q}{v}$,
\end{center}imply that there exist the limits
\begin{align}\label{k,l}
\textstyle \kappa :=\lim_{r\to 0} \frac{X(\mathrm{ln}r)}{r^2}=\frac{\eta^p}{ \lambda N\xi}, \qquad 	\ell:=\lim_{r\to 0} \frac{Y(\mathrm{ln}r)}{r^2}=\frac{\xi^q}{ \lambda N\eta},
\end{align}
and so
\begin{align}\label{Phi}
\textstyle \Phi (-\infty):=\lim_{t\to -\infty}\frac{X}{Y}(t)=\frac{\kappa}{\ell}=\frac{\eta^{p+1}}{\xi^{q+1}}, \quad \textrm{for } \,\Phi (t)=\Phi_{\Gamma}(t):=\frac{X(t)}{Y(t)} .
\end{align}

As a consequence of \eqref{Phi}, for any fixed pair of initial values $(\xi,\eta)$ it holds
\begin{align}\label{SZ xi,eta}
\textstyle \eta=c_\Gamma\, \xi^{\frac{q+1}{p+1}}, \quad \textrm{ where } \; c^{p+1}_\Gamma =\Phi(-\infty).
\end{align}
We note that $c_\Gamma^{p+1}$ gives the slope (of $X$ with respect to $Y$) of the projection of the regular trajectory $\Gamma$ on the plane $(X,Y)$, since $X(-\infty)=Y(-\infty)=0$.

Recall that the rescaled solutions $(u_\gamma,v_\gamma)$ defined in Remark \ref{rescaling}, by \eqref{X,Y,Z,W sem a}, are associated with the functions
\begin{align}\label{rescaling X,Y,Phi}
\textstyle X_\gamma (t)=X( t+\mathrm{ln} \gamma), \;\; Y_\gamma (t)= Y(t+\mathrm{ln} \gamma), \;\; Z_\gamma (t)=Z( t+\mathrm{ln} \gamma), \;\; Z_\gamma (t)= W(t+\mathrm{ln} \gamma) ,
\end{align}
whose corresponding trajectory in $\K$ is still $\Gamma$ since the system \eqref{DS+} is autonomous.
This implies that
\begin{align}\label{Phi gamma}
\textstyle\Phi_\gamma(-\infty)=\Phi (-\infty), \quad\textrm{ for }\;\;\Phi_\gamma=\frac{X_\gamma}{Y_\gamma},
\end{align}
In the case of $\M^-$, \eqref{k,l}--\eqref{Phi gamma} hold in the same way just replacing $\lambda$ by $\Lambda$ in \eqref{k,l}.

\begin{proof}[Proof of Theorem \ref{teo uniqueness}]
Let $(u,v)$, $(\bar{u},\bar{v})$ be two radial solution pairs of \eqref{LE}, either in $\rN$, or in the ball $B_R$ when $R<\infty$; in the latter case they also satisfy \eqref{H Dirichlet}. These are solutions of \eqref{shooting}, which are positively defined in the maximal radius $R$, where $R\le +\infty$.
One may choose $\gamma>0$ such that 
$u(0)=\gamma^\alpha \bar u(0)=\bar u_{\gamma} (0)=\xi$, 
where
$\bar{u}_{\gamma}(r)=\gamma^\alpha \bar u(\gamma r)$,  $\bar{v}_{\gamma}(r)=\gamma^\beta \bar v(\gamma r)$, for $\alpha, \beta$ as in \eqref{def alpha, beta}. Then $(\bar u_{\gamma},\bar v_{\gamma})$ is a positive solution pair of \eqref{shooting} for $R/\gamma$. Set $\frak m=\min\{R,R/\gamma\}$.
 
\smallskip

	\textit{Step 1)} $u(0)=\bar u_\gamma (0)$ yields $\bar v_\gamma(0) = v(0)$.

\smallskip

We suppose by contradiction that $\bar v_\gamma(0) < v(0)$; for the other inequality it is analogous.
If either $I=(0,\frak m]$ if $R<\infty$, or $I=(0,+\infty)$ if $R=\infty$, we first claim that
\begin{align}\label{claim ordering}
\textrm{$\bar v_\gamma(r) < v(r)$ \; for all $r\in I$.}
\end{align}

Indeed, assume on the contrary that there exists $a \in I$ such that $\bar{v}_\gamma - v<0$ in $[0, a)$ and $(\bar{v}_\gamma – v)(a)= 0$.
Let us first show this implies
 $\bar{u}_\gamma-u > 0$ in $(0, a]$ 
(here we do not include $0$ since $\bar{u}_\gamma (0)=u(0)$). 
	
	If this was not true, then there would exist $b \in (0, a]$ such that $(\bar{u}_\gamma-u)(b) \le  0$. Since 
	\begin{align}\label{conta u,wgamma}
	\textrm{$\mathcal{M}^+(D^2(\bar{u}_\gamma-u))\ge \mathcal{M}^+(D^2\bar{u}_\gamma)-\mathcal{M}^+(D^2 u) =v^p-\bar{v}_\gamma^p>0 $\, in $B_b$, \;\;  $\bar{u}_\gamma-u \le 0$\, on $\partial B_b\,$,}
	\end{align} 
then the maximum principle (MP) implies $\bar{u}_\gamma - u \le  0$ in $B_b$, and the strong maximum principle (SMP) yields $\bar{u}_\gamma - u <  0$ in $B_b$. We refer for instance \cite{BardidaLio, CafCab, Quaas2004} on properties of Pucci operators, MP and SMP. But this contradicts our initial assumption $\bar{u}_\gamma(0)=u(0)$.
	
	Now, by using the other equation we derive
	\begin{center}
		$\mathcal{M}^+(D^2(v-\bar{v}_\gamma))\ge \mathcal{M}^+(D^2v)-\mathcal{M}^+(D^2 \bar{v}_\gamma) =\bar{u}_\gamma^q-u^q\ge 0 $\; in $B_a$, \quad  $v-\bar{v}_\gamma = 0$\; on $\partial B_a\,$,
	\end{center} 
	then again MP and SMP (since $\bar{u}_\gamma>u$ out of $0$) give us $v-\bar{v}_\gamma  < 0$ in $B_a$, which in turn contradicts $\bar{v}_\gamma<v$ in $B_a$, and the claim \eqref{claim ordering} is proved.
	
Now we claim that, for the interval $I$ as above, we have
\begin{align}\label{claim2 ordering}
\textrm{$\bar{u}_\gamma> u$ in $I$.}
\end{align}

Otherwise, there would exist $b\in I$ ($b\neq 0$), such that  $\bar{u}_\gamma -u\le 0$ at $r=b$.  Then it implies \eqref{conta u,wgamma} (recall that $\bar{v}_\gamma<v$ in $I$ holds due to the conclusion of Claim \ref{claim ordering}). Thus we derive $\bar{u}_\gamma\le u$ in $B_b$ by MP. Again $\bar{u}_\gamma < u$ in $B_b$ by SMP (since the strictly inequality is in force in \eqref{conta u,wgamma}); in particular it holds at $0$.  But this contradicts $\bar{u}_\gamma (0) =u(0)$, and so \eqref{claim2 ordering} is proved.

\smallskip

(i) If $R=\infty$, then for  $\varepsilon>0$ we set $U=\bar{u}_\gamma -u-\varepsilon$. Since $u,\bar{u}_\gamma\to 0$ as $r\to \infty$ (cf.\ Corollary \ref{cor decay u,v}), then in particular there exists a large $b>0$ such that $U\le 0$ for all $|x|\geq b$, and as in \eqref{conta u,wgamma},
\begin{center}
	$\mathcal{M}^+(D^2 U)>0$\; in $B_b$, \qquad $U\leq 0$\; on $\partial B_b$. 
\end{center}
Hence MP yields $U\le 0$ in $B_b$, and by letting $b\to \infty$ we get $\bar{u}_\gamma-u\le \varepsilon$ in $\rN$. Since $\varepsilon$ is arbitrary, then $\bar{u}_\gamma \le u$ in $\rN$, which contradicts \eqref{claim2 ordering}.

\smallskip

(ii) If $R<\infty$, then $\sigma=(\bar{v}_\gamma - v)(\frak m)<0$. But since
\begin{center}
	$\sigma =- v(R/\gamma)<0$ \, if $ \gamma> 1$, \qquad $\sigma =0$ \, if $\gamma=1$, \qquad $\sigma =\bar{v}_\gamma(R)>0$ \, if $\gamma<1$,
\end{center}
then $\gamma > 1$. Next, by Claim \ref{claim2 ordering} we have $\bar{u}_\gamma-u > 0$ in $(0, \frak m]$. So $\theta=(\bar{u}_\gamma-u)(\frak m)>0$, while
\begin{center}
	$\theta =- u(R/\gamma)<0$ \, if $ \gamma> 1$, \qquad $\theta =0$ \, if $\gamma=1$, \qquad $\theta =\bar{u}_\gamma (R)>0$ \, if $\gamma<1$.
\end{center}
Therefore $\gamma <1$, which is impossible. This proves that $\bar{v}_\gamma(0) = v(0)$.

\vspace{0.2cm}

\textit{Step 2)} Equal shootings in \eqref{shooting} lead to equal solutions.

\smallskip

Note that if either $R=\infty$ or $R<\infty$ with $p,q\ge 1$, then the uniqueness of the ODE problem \eqref{shooting} is a consequence of Lipschitz continuity; in particular \textit{(i)} is proved. Assume then $R<\infty$ in the general superlinear regime $pq>1$. 

Recall we have chosen $\gamma>0$ such that $\xi=\gamma^\alpha\bar{\xi}$. Next, by Step 1 we have obtained $\eta=\gamma^\beta \bar{\eta}$. Let $\Gamma=(X,Y,Z,W)$ and $\bar \Gamma=(\bar X,\bar Y,\bar Z,\bar W)$ be the trajectories associated to $(u,v)$ and $(\bar u,\bar v)$ respectively, with $c_\Gamma$ and $ c_{\bar\Gamma}$ being the corresponding slopes of their projections on the plane $(X,Y)$. 
Further, $\bar \Gamma$ is also the corresponding trajectory to $(\bar u_\gamma, \bar v_\gamma)$, by \eqref{rescaling X,Y,Phi} since the system is autonomous.

Now we infer that $\Gamma=\bar{\Gamma}$. Indeed, if $\bar \kappa$ and $\bar \ell$ are the constants in \eqref{k,l} related to $\bar X$, $\bar Y$, then
$\bar \kappa=\kappa \gamma^{\alpha-\beta p}$ and $\bar \ell=\ell \gamma^{\beta-\alpha q}$, 
since $\xi=\gamma^\alpha\bar{\xi}$ and $\eta=\gamma^\beta \bar{\eta}$.
Hence ${c}_{\bar\Gamma}=c_\Gamma$, which means that the projections of $\Gamma$ and $\bar \Gamma$ on the plane $(X,Y)$ coincide.
Since the unstable manifold at $N_0$ is a graph on the variables $(X,Y)$ (see also Remark \ref{def local manifold N0} ahead), then the trajectories $\Gamma$ and $\bar \Gamma$ are the same.

So we deduce that $\gamma=1$ and $(u,v)=(\bar u,\bar v)=(\bar u_{\gamma},\bar v_{\gamma})$ in $B_R$. 
Since any positive solution of \eqref{LE}, \eqref{H Dirichlet} is radially symmetric by \cite{EG}, item \textit{(ii)} is proved.
\end{proof}

\subsection{Blow-up analysis of regular solutions}

As far as regular trajectories are concerned, the conclusion of Proposition \ref{AP bounds} (ii) is always verified. On the other hand, if (i) does not hold for such a trajectory $\tau$, it means that at least one of $X,Y$ blows up in finite forward time. 
In the sequel we characterize the blow-up of regular trajectories.

\begin{rmk}\label{def local manifold N0}
The local unstable manifold at the point $N_0$, which we denote by $\mathcal{W}_u(N_0)$, is a two dimensional graph of smooth functions $\varphi (X,Y)$, $\psi (X,Y)$, see Proposition \ref{local study stationary points}(2). 
Namely, for some $\rho>0$, set $\mathcal{U}(N_0)=B_\rho(0,0)\backslash\left\{  (0,0)\right\}\subset\real^2$, and write 
\begin{center}
	$\mathcal{W}_u(N_0)=\{\, (X,Y,\varphi(X,Y),\psi(X,Y))\in\real^4\;: \; (X,Y)\in \mathcal{U}(N_0)\,\}$. 
\end{center}
Let $\Gamma_{x,y}=(X,Y,Z,W)$ be the unique trajectory passing through the point $(x,y,\varphi(x,y),\psi(x,y))$ in $\mathcal{W}_{u}(N_0)$ at time $t=0$, where $(x,y)\in \mathcal{U}(N_0)$. Set $T_*=T_*(x,y)$ so that $(-\infty,T^*]$ is the maximal interval of existence for $\Gamma_{x,y}$.
\end{rmk}

We split the set\, $\mathcal{U}(N_0)$ as the disjoint union
$\mathcal{U}(N_0)=\mathcal{N}_1\cup\mathcal{N}_{2}	\cup\mathcal{D}\cup \mathcal{G}_{N_0}$, where \vspace{-0.15cm}
\[
\begin{array}
[c]{c}
\mathcal{N}_{1}=\{  (x,y)\in\mathcal{U}(N_0):x,y>0,\,\textrm{ for }\, \Gamma_{x,y}, \, \lim_{t\rightarrow T_*}X(t)=+\infty, \,Y(T_*)<+\infty \}  ,\vspace{0.08cm}\\
\mathcal{N}_{2}=\{  (x,y)\in\mathcal{U}(N_0):x,y>0,\,\textrm{ for }\, \Gamma_{x,y}, \, X(T_*)<+\infty, \,\lim_{t\rightarrow T_*}Y(t)=+\infty  \},\vspace{0.08cm}\\
\mathcal{D}=\{  (x,y)\in\mathcal{U}(N_0):x,y>0,\,\textrm{ for }\, \Gamma_{x,y}, \, \lim_{t\rightarrow T_*}X(t)=\lim_{t\rightarrow T_*}Y(t)=+\infty\} ,\vspace{0.08cm}\\
\mathcal{G}_{N_0}=\{  (x,y)\in\mathcal{U}(N_0):x,y>0,\,\, \Gamma_{x,y}  \textrm{ stays in the box $\mathcal{B}_\pm$ of Proposition \ref{AP bounds} }\} .
\end{array}\vspace{-0.1cm}
\]

\begin{prop}\label{Th1.1 lemma} 
$\mathcal{N}_1$ and $\mathcal{N}_2$ are open nonempty subsets. Moreover, $\mathcal{N}_1$ contains a neighborhood of the $X$--axis, and $\mathcal{N}_2$ contains a neighborhood of the $Y$--axis. Further,  $\mathcal{D}\cup \mathcal{G}_{N_0}= \overline{\mathcal{N}}_1\cap \overline{\mathcal{N}}_2 \neq \emptyset$.
\end{prop}

The proof of Proposition \ref{Th1.1 lemma} is similar to  \cite[Theorem 1.1]{BV} once the correct correspondence with the case of the Pucci-Lane-Emden system \eqref{LE} is established. Since this result plays a pivotal role in our analysis, and for reader's convenience, we provide some details in what follows.

\begin{proof}
In order to fix the ideas we consider the operator $\M^+$, since for $\M^-$ it is analogous.

We first infer that if $(x,y)\not\in \mathcal{G}_{N_0}$, then
	\begin{align}\label{X+Y to infty}
\textstyle \lim_{t\rightarrow T_*}(X(t)+Y(t))=+\infty.
	\end{align}
Indeed, Proposition \ref{AP bounds} yields $Z,W\in (0,\lambda N)$ for regular orbits. 
	Now, if both components $X,Y$ were bounded, then by the classical ODE theory we would have $\Gamma_{x,y}=(X,Y,Z,W)$ defined for every time $t$, which contradicts the hypothesis.
	Thus, at least one of $X,Y$ is unbounded.
	Hence there exists a first time $T$ so that $X(T)=\tilde{N}_+-2$ or $Y(T)=\tilde{N}_+-2$, $T<T_*$. 
Recall that if $X(T)=\tilde{N}-2$ then $\dot{X}>0$ from $T$ on by Proposition \ref{prop flow}(ii); the same for $Y$.

\vspace{0.08cm}

	In terms of solutions $(u,v)$ of \eqref{P radial m},  $\mathcal{N}_{1}$ is the set where $u$ vanishes before $v;$ similarly for $\mathcal{N}_{2}.$
	Further, from $\mathcal{D}$ we obtain the set of solutions $(u,v)$ of \eqref{P radial m} with $u$ and $v$ vanishing at the same time $R_*=e^{T_*}$ (see Remark \ref{Remark 2F}), so Hopf lemma yields
	\begin{align}\label{eq Hopf th1.1}
	\textstyle{\lim_{r\rightarrow R_*^-}\,\frac{u(r)}{(r-R_*)u^\prime (r)}=\lim_{r\rightarrow R_*^-}\,\frac{v(r)}{(r-R_*)v^\prime (r)}=1, \quad \textrm{then \;\;} \lim_{t\rightarrow		T_*^-}\,\frac{X}{Y}=1. }
	\end{align}
	
Let $\bar{x}\in\left( 0,\rho\right)$. Let us consider the borderline point $(\bar x, 0)$ of $\{(x,y)\in\mathcal{U}(N_0): x,y>0\}$, where $\Gamma_{\bar{x},0}=(\bar{X},\bar{Y}, \bar{Z}, \bar{W})$ is the trajectory with $\bar{Y}\equiv 0$ passing through the point $(\bar{x},0,\varphi(\bar{x},0),\psi(\bar{x},0))\in \mathcal{W}_u$ at time $t=0$. Let us see that $\bar{X}$ blows up in finite time and  $\Gamma_{\bar{x},0}$ does not correspond to a regular solution $u,v$ of \eqref{P radial m}.

Note that $Y\equiv 0$ means that $Y(t)=-\frac{\rmd }{\rmd  t}(\mathrm{ln}(v(e^t)))\equiv 0$, and so $v$ is a positive constant. Also, the system \eqref{DS+} for $Y\equiv 0$ becomes
\begin{align}\label{sistema Y=0}
\dot{X} = \;X \; [\,X+1-M_+ (\lambda(N-1)-Z )\,], \quad 
\dot{Z} = \;Z\; [\,1 +M_+ (\lambda(N-1)-Z )\,], \\
\;\dot{W} =\;W\; [\,1-qX +M_+ (\lambda(N-1)-W )\,].\nonumber
\end{align}

The equation for $Z$ in \eqref{sistema Y=0} is autonomous, whose RHS is continuous and positive at $\pi_{3,\lambda}$. Thus a qualitative ODE analysis unveils that either: 

(a) $Z\in (0,\lambda N)$ is defined for all time, and it is forward increasing;

(b) $Z\equiv \lambda N$ is a stationary orbit for all time;

(c) or $Z>\lambda N$ blows up in finite backward time, and it is forward decreasing for all time. 

We infer that situation (c) is not admissible. In fact, a point of blow-up for $Z$ would produce a positive radius $r_0>0$ at which $u^\prime (r_0)=0$, and so $X(\mathrm{ln} r_0)=0$. But this is impossible since the projection of the trajectory $\Gamma_{\bar{x},0}$ on the plane $(X,Y)$ lies on the $X$ axis; i.e.\ it starts at $(0,0)$ and is increasing in the $X$ direction.

If (b) is true, then $\dot X=X(X+2)>0$, it is easy to see that $X$ blows up in finite forward time, and we are done.
Assume then (a) occurs, then the explicit expression for $Z$ is given by 
\begin{center}
	$Z(t)=\frac{\lambda N c_1 e^{N(t-t_1)} }{1+ c_1 e^{N(t-t_1)} }$,  \quad where $c_1=\frac{Z_1}{ \lambda N- Z_1 }>0$, \quad for $Z_1=Z(t_1)$.
\end{center}
By replacing this into the definition of the variable $Z$ (recall that $v\equiv c$) yields an explicit expression for $u$ in which $u(0)=+\infty$ and $u(+\infty)=-\infty$. Hence by the Mean Value theorem there exists a radius $R>0$ at which $u(R)=0$, and so $X$ blows up in finite time, thus $(\bar{x},0)\in \mathcal{N}_1$.
A simple computation shows that such a pair $(u,c)$ does not satisfy our PDE system \eqref{LE}.

The next step is to show that $\mathcal{N}_1$ contains a neighborhood of $(\bar x, 0)$. Let us prove that a neighborhood of regular trajectories $\Gamma_{x,y}$ with $x,y>0$ in $\mathcal{W}_u(N_0)$ blows up only in $X$ at $T_*(x,y)$ whenever $Y$ is taken sufficiently small.

Set $A:= 2N-\tilde{N}>0$, pick up $\varepsilon\in (0,A)$ and $\eta>0$ such that $B_\eta(\bar{x},0)\subset B_\rho(0,0)$. Since $\lim_{t\to T_*}\bar{X}(t)=+\infty$, by continuity of the ODE system with respect to the initial conditions, for any $(x,y)\in B_\eta(\bar{x},0)$ with $y>0$ and $\Gamma_{x,y}=(X,Y,Z,W)$ as above, there exists 
$T_{\varepsilon}<T_*$
such that $X(T_{\varepsilon})=2A$, and $Y(t)\in (0,\varepsilon)$ for all $t\leq	T_{\varepsilon}$. Note that $X$ is increasing from $T_\varepsilon$ on. 
Now, define $\Phi=X/Y$ in $(-\infty,T_*)$. By using 
\begin{align}\label{cota M Z,W}
\textstyle M_+(\lambda(N-1)-Z)\le N-1-\frac{Z}{\Lambda},  \quad
M_+(\lambda (N-1)-W)\ge \tilde{N}-1-\frac{W}{\lambda}, 
\end{align}
together with $Z\geq 0$ and $W\leq \lambda N$, we obtain
	\begin{align}\label{eq1 th1.1}
	\textstyle \frac{\dot{\Phi}}{\Phi}=X-Y-M_+(\lambda(N-1)-Z)+M_+(\lambda (N-1)-W)\geq X-Y-A,
	\end{align}
	so $\dot{\Phi}(T_{\varepsilon})>0$. 
Set $\theta=\sup\{\, 
 t\in (T_{\varepsilon},T_*]
:\,\dot{\Phi}>0\, \textrm{ in }(T_\varepsilon,t)\, \}$, then 
	\begin{align}\label{eq2 th 1.1}
	\textrm{$\Phi(t)\ge \Phi\left(  T_{\varepsilon}\right)  \ge 2A/\varepsilon>2$\quad for all $t\in (T_\varepsilon,\theta]$. }
	\end{align}
Here $X,Y$ do not blow up in $(T_\varepsilon,T_*)$ by definition of $T_*=T_*(x,y)$ and construction of $T_\varepsilon$.

If we had $\theta<T^*$ then $\dot{\Phi}(\theta)=0$, so \eqref{eq1 th1.1} would imply $X(\theta)  \leq Y(  \theta)  +A$. Also, $Y( \theta)  <X\left(	\theta\right)  /2$ by \eqref{eq2 th 1.1}. Thus $X(\theta)<X(\theta)  /2+A$, which contradicts the fact that $X(\theta)\geq X(T_\varepsilon)=2A$. 
Hence $\theta=T^*$. If $\lim_{t\rightarrow T^{\ast}}Y(t)=\infty$ then $(x,y)\in \mathcal{D}$ and so $\lim_{t\rightarrow T_*}\Phi(t)=1$ by \eqref{eq Hopf th1.1}; but this is impossible by \eqref{eq2 th 1.1}. 
Therefore $(x,y)\in\mathcal{N}_{1}$, for any $(x,y)\in B_\eta (\bar x, 0)$.
This shows that a neighborhood of $\Gamma_{\bar{x},0}$ is contained in $\mathcal{N}_{1}$. Reasoning similarly, one proves that $\mathcal{N}_1$ is open.
By exchanging the roles of $X,Y$ one sees that $\mathcal{N}_2$ is nonnempty and open. 

Now, since $\mathcal{U}(N_0)$ is connected, then either $\mathcal{D}$ or $\mathcal{G}_{N_0}$ is nonempty.
\end{proof}

\begin{rmk}
The solutions associated with trajectories on the axes $X\equiv 0$ and $Y\equiv 0$ in the plane $(X,Y)$ can be seen as ``infinite shootings", namely  $X\equiv 0 \leftrightarrow ``\eta=+\infty"$ and $Y\equiv 0 \leftrightarrow ``\xi=+\infty"$.
\end{rmk}

\begin{proof}[Proof of Theorem \ref{teo uniqueness trajectory}]
Let us show that $\overline{\mathcal{N}}_1\cap \overline{\mathcal{N}}_2$ contains a unique trajectory. In terms of the second order PDE problem, this means that $\mathcal{G}$ is empty if and only if $\C$ is nonempty, and viceversa.	
	
We fix a pair $(p,q)$ with $p,q>0$ and $pq>1$.
By Proposition \ref{Th1.1 lemma}, any point of $\mathcal{D}$ or $\mathcal{G}_{N_0}$ must lie in the intersection $\overline{\mathcal{N}}_1\cap \overline{\mathcal{N}}_2$.
Next we infer that $\mathcal{D}$ consists of a unique curve given by the projection of a unique trajectory, and the same holds for $\mathcal{G}_{N_0}$. 

Indeed, recall that there is a biunivocal correspondence between $\mathcal{U}(N_0)$ and the set of orbits in $\real^4$ around $N_0$ by Remark \ref{def local manifold N0}; in particular in $\mathcal{U}(N_0)$ there cannot be orbits intersection. Formally, the uniqueness of projections is a consequence of uniqueness of trajectories. We infer that the latter for $\mathcal{D}$ and $\mathcal{G}_{N_0}$, in turn, comes from uniqueness of solutions given in Theorem \ref{teo uniqueness}.

Indeed, if we had two trajectories $\Gamma$ and $\tilde{\Gamma}$ defined in $\rN$, by the one to one correspondence between orbits of \eqref{DS+} and solutions of \eqref{LE} given by \eqref{X,Y,Z,W sem a} and \eqref{def u via X,Z}, these would be associated with two solutions $(u,v)$ and $(\tilde{u},\tilde{v})$ in $\rN$. Then Theorem \ref{teo uniqueness}(i) yields $(\tilde{u},\tilde{v})=(u_\gamma,v_\gamma)$ in $\rN$. Since the system is autonomous, $\Gamma=\Gamma_\gamma=\tilde{\Gamma}$, see \eqref{rescaling X,Y,Phi}. 
 By a similar reasoning, any two trajectories corresponding to two regular solutions in the ball, even for different radii, represent the same trajectory by Theorem \ref{teo uniqueness}(ii).
 
 Therefore, $\overline{\mathcal{N}}_1\cap \overline{\mathcal{N}}_2\cap \mathcal{G}_{N_0}$ consists of at most one trajectory in $\mathcal{U}(N_0)$, and the same is true for $\overline{\mathcal{N}}_1\cap \overline{\mathcal{N}}_2\cap \mathcal{D}$.  This gives us at most two connected components for $\overline{\mathcal{N}}_1\cap \overline{\mathcal{N}}_2$.
Hence, we only need to prove that $\mathcal{N}_1$ and $\mathcal{N}_2$ are connected sets, which translates into saying that we do not have corresponding solutions both in a ball and in $\rN$ simultaneously.

To see this, we use the fact that the ordering of the connected components around the axes $X,Y$ is already prescribed, in the sense that one displays $\mathcal{N}_1$ near $Y\equiv 0$, and $\mathcal{N}_2$ near $X\equiv 0$. Thus $\overline{\mathcal{N}}_1\cap \overline{\mathcal{N}}_2$ must have an odd number of connected components.
This means that if $\overline{\mathcal{N}}_1\cap \overline{\mathcal{N}}_2$ had more than one connected component, then it would need to have at least three (see Figure \ref{Fig manifoldN0}); but this is impossible since we have at most two of them.
	\begin{figure}[!htb]\centering	\includegraphics[scale=0.45]{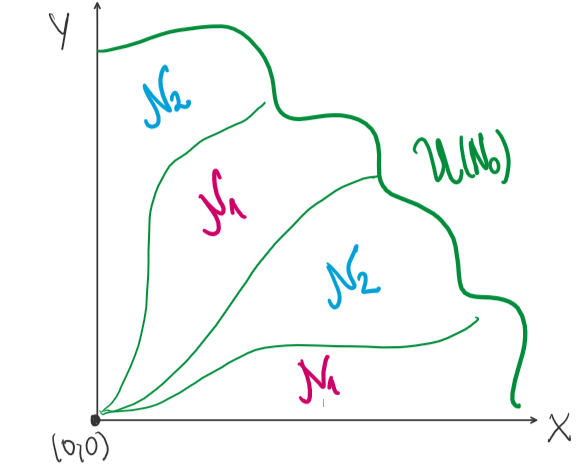}	
		\caption{The projection of the stable manifold on the plan $X,Y$ in the case that $\overline{\mathcal{N}}_1\cap\overline{\mathcal{N}_2}$ has three connected components.}\label{Fig manifoldN0}
	\end{figure}
\end{proof}

\begin{defin}\label{def Gammap,q}
For each $(p,q)$ with $p,q>0$, $pq>1$, by Theorem \ref{teo uniqueness trajectory}, we define $\Gamma_{p,q}\in \overline{\mathcal{N}_1}^{p,q}\cap \overline{\mathcal{N}_2}^{p,q}$ as the unique regular trajectory exiting $N_0$ such that either $(p,q)\in \C$ or $(p,q)\in \mathcal{G}$.
\end{defin}

\begin{rmk}\label{rmk Gamma p,q cont}
The mapping $(p,q)\mapsto \Gamma_{p,q}$ is continuous. Indeed, this is a consequence of \eqref{eigenvector N0} (see the proof of Proposition \ref{local study stationary points}(2)), since in a neighborhood of $N_0$, $Z$ and $W$ are continuous functions of $X,Y,p,q$.
\end{rmk}

\begin{prop}\label{C is open}
$\C$ is open.
\end{prop}

\begin{proof}
For simplicity, let us consider the operator $\M^+$; for $\M^-$ it is similar.

Let $(p_0,q_0)\in \C$. 
Then there exists a regular trajectory $\Gamma_0$ of \eqref{DS+}$_{p_0,q_0}$ that crosses both $L_X^\pm$ and $L_Y^\pm$ before $X,Y$ blowing up together at the finite time $T_0^*$. Say $\Gamma_0$ passes through the point $(x_0,y_0,\varphi (x_0,y_0),\psi (x_0,y_0))\in \mathcal{W}_u(N_0)$  at time $t=0$, for some $(x_0,y_0)\in \mathcal{U}(N_0)$, see Remark \ref{def local manifold N0}. 
By Theorem \ref{teo uniqueness trajectory} one knows that positive radial solutions of \eqref{LE}$_{(p_0,q_0)}$ in $\rN$ do not exist.

Fix $(p,q)$ near $(p_0,q_0)$. 
By ODE properties of continuity with respect to data and parameters, any regular trajectory $\Gamma_{x,y}=(X,Y,Z,W)$ of \eqref{DS+}$_{p,q}$, with $(x,y)$ close to $(x_0,y_0)$, namely $(x,y)\in B_{\delta}(x_0,y_0)$, also has to cross both $L_X^\pm$ and $L_Y^\pm$, in such a way that both $X$ or $Y$ blow up in finite time by Proposition \ref{AP bounds}, namely at $T^*_X$ and $T^*_Y$. 

Moreover, $\Gamma_{p,q}$ (see Definition \ref{def Gammap,q}) passing through the point $(x,y)$ at $t=0$ is so that $(x,y)\in B_{\delta}(x_0,y_0)$ for $(p,q)$ sufficiently close to $(p_0,q_0)$. 
So $(p,q)\not\in \mathcal{G}$, and $(p,q)\in \C$ by Theorem \ref{teo uniqueness trajectory}. 
\end{proof}

\subsection{Blow-up analysis of fast decaying solutions}

In this section we pay a special attention to fast decaying trajectories.

\begin{defin}\label{def fast decay solu}
	The set of fast decaying solutions, which we denote by $\F$, is the set of $(p,q)$ with $p,q>0$ and $pq>1$ such that solutions of \eqref{P radial m} or \eqref{P radial m-} defined in $\rN$ have their respective trajectories of \eqref{DS+} or \eqref{DS-} converging to either $A_0$, $P_0$ or $Q_0$.
\end{defin}

We stress that fast decaying trajectories are in one to one correspondence (up to scaling) with fast decaying solutions settled in Definition \ref{def fast}, accordingly to Proposition \ref{local study stationary points}(3)--(5); see also \cite[proof of Proposition 3.9]{MNPscalar}.

Set $\frak P=(x_0,y_0,z_0,w_0)$ as either $A_0$, $P_0$ or $Q_0$. Here,
\begin{center}
$(z_0,w_0)=(0,0)$ \; if $\frak P=A_0$ for $p,q>\frac{\tilde N_\pm}{\tilde N_\pm -2}$;
	
	$(z_0,w_0)=(0,\Lambda (\tilde{N}_\pm -q(\tilde{N}_\pm-2) ) )$ \; if $\frak P=P_0$ and $p>\frac{\tilde N_\pm}{\tilde N_\pm -2}$;
	
	$(z_0,w_0)=(\Lambda (\tilde{N}_\pm -p(\tilde{N}_\pm-2) ), 0 )$ \; if $\frak P=Q_0$ and $q>\frac{\tilde N_\pm}{\tilde N_\pm -2}$.
\end{center}

\smallskip

When $A_0\neq P_0$ and $A_0\neq Q_0$, we recall that $\mathcal{W}_s(\frak P)$ is locally a two dimensional graph of smooth functions $\varphi (Z,W)$, $\psi (Z,W)$; see Proposition \ref{local study stationary points} (3)--(5), \eqref{eigenvector A0}, and the proof item (4) there. 
For some $\rho>0$, set $\mathcal{U}(\frak P)=B_\rho(z_0,w_0)\backslash\left\{  (z_0,w_0)\right\}\subset\real^2$, and write \begin{center}
	$\mathcal{W}_s(\frak P)=\{\, (\varphi(Z,W),\psi(Z,W),Z,W)\in\real^4\;: \; (Z,W)\in \mathcal{U}(\frak P)\,\}$.
\end{center}
 
Given $(z,w)\in \mathcal{U}(\frak P)$, let $\Upsilon_{z,w}=(X,Y,Z,W)$ be the unique trajectory passing through the point $(\varphi(z,w),\psi(z,w),z,w)$ in $\mathcal{W}_{s}(\frak P)$ at time $t=0$. Then set $T_*=T_*(z,w)$ so that $[T^*,+\infty)$ is maximal interval of existence for $\Upsilon_{z,w}$.

At the collision points $A_0=P_0$ and $A_0=Q_0$, a trajectory in $\K$ enters by a center manifold of dimension one, while its stable manifold is also one dimensional, i.e.\
\begin{center}
$\mathcal{W}_c(\frak P)=\{\, (X(W),Y(W),Z(W),W)\in\real^4\,\}$, \quad $\mathcal{W}_s(\frak P)=\{\, (X(Z),0,Z,0)\in\real^4\,\}$.
\end{center}
Thus, any trajectory converging either to $A_0= P_0$ or $A_0= Q_0$ is a local graph on the variables $(Z,W)$, which again can be represented by 
\begin{align}\label{Ws,Wc}
\textrm{$\{\, \Upsilon_{z,w}= (\varphi(Z,W),\psi(Z,W),Z,W)\in\real^4\,\}$ \; near either $A_0= P_0$ or $A_0= Q_0$,}
\end{align}
whose projection we still denote by $(z,w)\in \mathcal{U}(\frak P)$ w.l.g.

\smallskip

Now, we split the set\, $\mathcal{U}(\frak P)$ as the disjoint union
$\mathcal{U}(\frak P)=\mathcal{A}_1\cup\mathcal{A}_{2}	\cup {\varSigma}\cup \mathcal{G}_{\frak P}$, where \vspace{-0.15cm}
\[
\begin{array}
[c]{c}
\mathcal{A}_{1}=\{  (z,w)\in\mathcal{U}(\frak P): z,w>0,\,\textrm{ for }\, \Upsilon_{z,w}, \, \lim_{t\rightarrow T_*}Z(t)=+\infty, \,W(T_*)<+\infty \}  ,\vspace{0.08cm}\\
\mathcal{A}_{2}=\{  (z,w)\in\mathcal{U}(\frak P):z,w>0,\,\textrm{ for }\, \Upsilon_{z,w}, \, Z(T_*)<+\infty, \,\lim_{t\rightarrow T_*}W(t)=+\infty  \},\vspace{0.08cm}\\
\varSigma=\{  (z,w)\in\mathcal{U}(\frak P):z,w>0,\,\textrm{ for }\, \Upsilon_{z,w}, \, \lim_{t\rightarrow T_*}Z(t)=\lim_{t\rightarrow T_*}W(t)=+\infty\} ,\vspace{0.08cm}\\
\mathcal{G}_{\frak P}=\{  (z,w)\in\mathcal{U}(\frak P):z,w>0,\,\, \Upsilon_{z,w}  \textrm{ stays in the box $\mathcal{B}_\pm$ of Proposition \ref{AP bounds} }\} .
\end{array}\vspace{-0.1cm}
\]

\smallskip

\begin{prop}\label{Th1.1 lemma A0} The sets $\mathcal{A}_1$ and $\mathcal{A}_2$ are nonempty and open. Also, $\mathcal{A}_1$ contains a neighborhood of the $Z$-axis, and $\mathcal{A}_2$ contains a neighborhood of the $W$-axis.
Further, 
$\mathcal{\varSigma}\cup \mathcal{G}_{\frak P}= \overline{\mathcal{A}}_1\cap \overline{\mathcal{A}}_2 \neq \emptyset$.
\end{prop}

\begin{proof} We fix the operator $\M^+$, since for $\M^-$ will be analogous. Our strategy is to show that $\mathcal{A}_1$ is nonempty and open. By symmetry it will follow that $\mathcal{A}_2$ is also nonnempty and open. Thus, since $\mathcal{U}(A_0)$ is connected, either $\varSigma$ or $\mathcal{G}_{A_0}$ will be nonempty.
	
	We first infer that if $(z,w)\not\in \mathcal{G}_{A_0}$, then
	\begin{align}\label{Z+W to infty}
	\textstyle \lim_{t\rightarrow T_*}(Z(t)+W(t))=+\infty.
	\end{align}
	Indeed, Proposition \ref{AP bounds} yields $X,Y\in (0,\tilde{N}_+-2)$ for fast decaying orbits. 
	Now, if both components $Z,W$ were bounded, then by classical ODE theory we would have $\Upsilon_{z,w}=(X,Y,Z,W)$ defined for every time $t$, which contradicts the hypothesis.
	Thus, at least one between $Z,W$ is unbounded.
	Hence there exists a first time $T$ so that $Z(T)=\lambda N$ or $W(T)=\lambda N$, $T>T_*$. 
	Recall that if $Z(T)=\lambda N$ then $\dot{Z}<0$ up to time $T$, by Proposition \ref{prop flow}(iii); the same for $Z$.
	
	\vspace{0.08cm}

	In terms of solutions $(u,v)$ of \eqref{P radial m},  $\mathcal{A}_{1}$ is the set where $u^\prime$ vanishes before $v^\prime;$ similarly for $\mathcal{A}_{2}.$
	Further, from $\varSigma$ we obtain the set of solutions $(u,v)$ of \eqref{P radial m} with $u^\prime$ and $v^\prime$ vanishing at the same time $R_*=e^{T_*}$, which gives us a solution in the exterior of a ball. 

\smallskip
	
We first consider $\frak P =A_0$, with $A_0\neq P_0$ and $A_0\neq Q_0$, and the operator $\M^+$. 

Step 1) $\mathcal{A}_1$ is nonempty.

We recall that the stable manifold at the point $A_0$ is also a graph of the variables $X,Y$, due to the tangent direction expressions in \eqref{eigenvector A0}. 
So we may consider the trajectory $\Upsilon_{\bar{z},\bar{w}}=\Upsilon_{\bar{x},\tilde{N}_+-2}=(\bar{X},\bar{Y}, \bar{Z}, \bar{W})$  with $\bar{Y}\equiv \tilde{N}_+-2$ passing through the point $(\bar x,\tilde{N}_+-2,\bar{z},\bar w)\in \mathcal{W}_s(A_0)$ at time $t=0$.
Let us see that $\bar{Z}$ blows up in finite time.
	
When $Y\equiv \tilde{N}_+-2$, the equation for $Z$ in the system \eqref{DS+} becomes independently expressed by
	\begin{align}\label{sistema X=N-2}
	\dot{Z} = Z\,  [\,1-p(\tilde{N}_+-2)+M_+ (\lambda(N-1)-Z )\,]=
 \begin{cases}
\,Z\,  [\,\tilde{N}_+ -p(\tilde{N}_+-2)-\frac{Z}{\Lambda} \,] \textrm{ \, in } R^+_{\lambda,Z}\\
\,Z\,  [\,N -p(\tilde{N}_+-2)-\frac{Z}{\lambda} \,] \textrm{ \, in } R^-_{\lambda,Z}.\\
\end{cases}
	\end{align}
Recall that $p>\frac{\tilde{N}_+}{\tilde{N}_+-2}$ since we are in the range of $(p,q)$ corresponding to the point $A_0$. Thus, in the region $R^-_{\lambda,Z}$ the stationary points $Z$ for the equation \eqref{sistema X=N-2} are $Z=\tilde{N}_+ -p(\tilde{N}_+-2)<0$ and $Z=0$. A qualitative analysis then unveils that solutions  $Z(t)>0$ are decreasing for all time, by converging to zero when $t\to +\infty $ and blowing up in finite backward time. 

On the other hand, in the region $R^+_{\lambda,Z}$ one finds a possible positive stationary point for \eqref{sistema X=N-2}, namely $Z=N-p(\tilde{N}_+-2)$.
In this case, an increasing orbit $Z(t)$ defined for $t\to -\infty$ also exists, but it is not admissible for our problem since at the concavity plane $\pi_{\lambda,Z}$ the equation \eqref{sistema X=N-2} produces
$\dot{Z}_{|_{ Z=\lambda (N-1)} }<0$.

Anyways, solutions are differentiable and decreasing for all time, by converging to zero when $t\to +\infty $, and blow up in finite backward time $T$. 
Note that $\Upsilon_{\bar{z},\bar w}$ does not correspond to a fast decaying positive solution $u,v$ of \eqref{LE}. Indeed, $Y\equiv \tilde{N}_+-2$ yields $v(r)={C}r^{2-\tilde{N}_+}$ and so $\M^+(D^2 v)=0$ in $R^-_{\lambda,W}$.

\medskip

Step 2) $\mathcal{A}_1$ is open.

Let us show that $\mathcal{A}_1$ contains a neighborhood of $(\bar z, \bar w)$, that is, all fast decaying trajectories $\Upsilon_{z,w}$ near $A_0$ in $\mathcal{W}_s(A_0)$ blow up only in $Z$ at $T_*(x,y)$ whenever $W$ is sufficiently small.  The proof that $\mathcal{A}_1$ is open will be a consequence of such argument.
	
We set $D:=q(\tilde{N}_+-2)+N-\tilde{N}_+\ge 1$, and choose $\varepsilon\in (0,\lambda D)$, $\eta>0$ so that $B_\eta(\bar{z},0)\subset B_\rho(0,0)$. Since $\lim_{t\to T_*}\bar{Z}(t)= +\infty$, by continuity of the ODE system with respect to data, for any $(z,w)\in B_\eta(\bar{z},0)$ with $w>0$ and $\Upsilon_{z,w}=(X,Y,Z,W)$, there exists 	$T_{\varepsilon}>T_*$ such that 
$Z(T_{\varepsilon}) =2\Lambda D$,
and $W(t)\in (0,\varepsilon)$ for all $t\geq	T_{\varepsilon}$. 
Note that $Z$ is decreasing up to $T_\varepsilon$ by Proposition \ref{prop flow}(iii). Define $\Psi=Z/W$ in $(T_*,+\infty)$. 
By using \eqref{cota M Z,W} in addition
to $X\leq \tilde{N}_+-2$ and $Y\geq 0$, one finds
\begin{align}\label{eq1 th1.1 A0}	\textstyle \frac{\dot{\Psi}}{\Psi}=qX-pY+M_+(\lambda(N-1)-Z)-M_+(\lambda (N-1)-W)\leq D - \frac{Z}{\Lambda}+\frac{W}{\lambda},	\end{align}	thus $\dot{\Psi}(T_{\varepsilon})<0$. 	
Set $\theta=\inf\{\, 	t\in [T_*,T_{\varepsilon})	:\,\dot{\Psi}<0\, \textrm{ in }(t,T_{\varepsilon})\, \}$, then 
\begin{align}\label{eq2 th 1.1 A0}	\textstyle \Psi(t)\ge \Psi\left(  T_{\varepsilon}\right)  \ge \frac{2\Lambda D}{\varepsilon}> \frac{2\Lambda}{\lambda}\quad \textrm{ for all $t\in [\theta,T_\varepsilon)$. }	\end{align}	Here $Z,W$ do not blow up in $(T_*,T_\varepsilon)$ by definition of $T_*=T_*(z,w)$ and construction of $T_\varepsilon$.

If we had $\theta>T^*$ then $\dot{\Psi}(\theta)=0$, so \eqref{eq1 th1.1 A0} yields $Z(\theta)  \leq \frac{\Lambda}{\lambda}W(\theta)  +\Lambda D <\frac{Z(\theta)}{2}+\lambda D$, by using \eqref{eq2 th 1.1 A0} at $t=\theta$, from which we deduce $Z(\theta)<2\Lambda D$.
But this contradicts $Z(\theta) \ge Z(T_\varepsilon)=2\Lambda D$.
Hence $\theta=T^*$. If $\lim_{t\rightarrow T^{\ast}}Z(t)=\infty$ then $(z,w)\in \Sigma_{A_0}$, and for $R_*=e^{T_*}$ it holds
\begin{align*}
\lim_{t\rightarrow T_*}\Psi(t)=\lim_{r\rightarrow R_*} \frac{v^pv^\prime}{u^q u^\prime}
=\lim_{r\rightarrow R_*}\frac{pv^{p-1}(v^\prime)^2 +v^pv^{\prime\prime}}{qu^{q-1}(u^\prime)^2 +u^q u^{\prime\prime}}
=\lim_{r\rightarrow R_*}\frac{pv^{p-1}(v^\prime)^2 -v^p\,\{\frac{(N-1)}{r}v^\prime+\frac{u^q}{\lambda} \} }{qu^{q-1}(u^\prime)^2 -u^q \{\frac{(N-1)}{r}u^\prime+\frac{v^p}{\lambda} \} }=1,
\end{align*}
by using L'Hospital rule and the fact $u,v$ are concave near $T^*$. But this contradicts \eqref{eq2 th 1.1 A0}, since $\Psi$ is decreasing. Hence $(z,w)\in\mathcal{A}_{1}$, for any $(z,w)\in \mathcal{U}(A_0)$.
Then a neighborhood of $\Upsilon_{\bar{z},\bar w}$ is contained in $\mathcal{A}_{1}$.

\medskip

Next we look at the colision point $\frak P=Q_0=A_0$ (the case $\frak P=P_0=A_0$ will be similar), in which we need to replace $\mathcal{W}_s(\frak P)$ by \eqref{Ws,Wc}.  
When $Y\equiv \tilde{N}_+-2$ and $p=\frac{\tilde N_+}{\tilde N_+ -2}$ the system \eqref{DS+} becomes
\begin{align*}
\textstyle \dot Z = -\frac{Z^2}{\Lambda} \;\; \textrm{ in } R^-_{\lambda , Z} \, , 
\quad
\dot Z = Z(N-\tilde N_+ - \frac{Z}{\lambda} ) \;\;\textrm{ in } R^+_{\lambda, Z} .
\end{align*}
A qualitative analysis on these autonomous equations in $Z$ gives us that the resulting trajectory in $R^-_{\lambda , Z}$ is forever decreasing and blows up in finite backward time.

Meanwhile, in $R^+_{\lambda , Z}$ we have a positive stationary point $Z=N-\tilde N_+$ which produces two types of orbits: one decreasing and blowing up in finite backward time, and another one increasing and defined for all backward time. The latter is not admissible because in the point $Z=\lambda (N-1)$ where the solution $\dot Z$ changes definition we have $\dot Z <0$.
Hence, the trajectory needs to blow up in finite backward time, from which $Z\to +\infty$ as $t\to t_0$ for some $t_0\in \real$.

Observe that when $Y\equiv \tilde{N}_+-2$, then $y\equiv 0$ in \eqref{L(Q0)} and the linearized direction that gives us $W\equiv 0$. In particular, the trajectory $\bar{\Upsilon}:=(\bar X, \tilde N_+-2, \bar Z, 0)$ satisfying $\bar{\Upsilon}=\Upsilon_{\bar x, \tilde N_+-2}$ is such that $(\bar x, \tilde N_+-2)\in \mathcal{A}_1$. As in Step 2 one finds that $\mathcal{A}_1$ is open.

\smallskip

Finally we treat the case $\frak P=Q_0\neq A_0$ (the corresponding $\frak P=P_0\neq A_0$ will be analogous). 

By the proof of Proposition \ref{local study stationary points} 4(i), the formulas \eqref{exp X,Z P0}, \eqref{eig P0 sigma2} with $c_2\neq 0$, $c_4\neq 0$, by properly changing the roles of $X$ and $Y$, $Z$ and $W$, $p$ and $q$) and the two main principal tangent directions gives us that $Z$ and $W$ only depend on $X,Y$. So the manifold around $Q_0$ is also a graph on the variables $X,Y$. Thus, one may consider the trajectory $\bar \Upsilon=\Upsilon_{\bar z, \bar w}=\Upsilon_{\bar x, \tilde{N}_+-2}=(\bar X, \bar Y, \bar Z, \bar W)$ such that $Y\equiv \tilde{N}_+-2$.
Arguing as in Steps 1 and 2 we get that $\mathcal{A}_1$ is nonempty and open.
\end{proof}

\section{Energy analysis}\label{section.1 energy}

Our first goal in this section is to prove our main existence and nonexistence results; see Figure \ref{Fig hyp2}. In terms of regular solutions, we will show that 
$\mathcal{R}_u^+\subset \mathcal{G}$ and $\mathcal{R}_d^+\subset \C$.

We define the following auxiliary hyperbolas for the study of the $\M^+$ operator.
\begin{equation}\label{Hu e Hd}
\overline{\sH}_+ \, : \;\; \frac{1}{p+1}+\frac{1}{q+1}=\frac{\tilde{N}_+ -2}{N},
\qquad
\underline{\sH}_+ \, : \;\; \frac{1}{p+1}+\frac{1}{q+1}=\frac{N-2}{\tilde{N}_+},
\end{equation}
We observe that the upper region $\mathcal{R}^+_u$ contains the hyperbola $\overline{\sH}_+$, and the down region $\mathcal{R}^+_d$ contains the hyperbola $\underline{\sH}_+$, see Figure \ref{Fig hyp2} where \eqref{Hu e Hd} are dashedly displayed. Here, $\underline{\sH}_+$ gives the asymptotic behavior of the upper boundary of $\mathcal{R}^+_d$ when either $p\to \infty$ or $q\to \infty$, and $\overline{\sH}_+$ gives the asymptotic behavior of the boundary of $\mathcal{R}^+_u$ when either $p\to \infty$ or $q\to \infty$. Note that for the region $\mathcal{R}_{s}^+$ in \eqref{region ABQ}, its upper boundary asymptotic is described through $\tH_+$ when either $p\to +\infty$ or $q\to +\infty$.

Let $\sigma$ be a positive parameter, and define the energy functional $E=E_{\sigma}(u,v): [0,\infty)\rightarrow \real$ of a solution $u,v$ of \eqref{P radial m} with $u^\prime, v^\prime<0$ as follows:
\begin{align*}
E(r)=
\begin{cases}
\,r^N 
\left(  u^\prime v^\prime  +\frac{1}{\lambda}\frac{v^{p+1}}{p+1} 
+\frac{1}{\lambda}\frac{u^{q+1}}{q+1}
+\frac{N}{p+1} \frac{v u^\prime}{r}  +\frac{N}{q+1} \frac{u v^\prime}{r} \right) &\textrm{ in }\, \{u^{\prime\prime}<0\}\cap \{v^{\prime\prime}<0\}, \medskip
\\
\,r^\sigma
\left(  u^\prime v^\prime  +\frac{1}{\lambda}\frac{v^{p+1}}{p+1} 
+\frac{1}{\Lambda}\frac{u^{q+1}}{q+1}
+\frac{N}{p+1} \frac{v u^\prime}{r}  +\frac{\tilde{N}_+}{q+1} \frac{u v^\prime}{r} \right) &\textrm{ in }\, \{u^{\prime\prime}<0\}\cap \{v^{\prime\prime}>0\}, \medskip
\\
\,r^\sigma
\left(  u^\prime v^\prime  +\frac{1}{\Lambda}\frac{v^{p+1}}{p+1} 
+\frac{1}{\lambda}\frac{u^{q+1}}{q+1}
+\frac{\tilde{N}_+}{p+1} \frac{v u^\prime}{r}  +\frac{N}{q+1} \frac{u v^\prime}{r} \right) &\textrm{ in }\, \{u^{\prime\prime}>0\}\cap \{v^{\prime\prime}<0\}, \medskip
\\
\,r^{\tilde{N}_+} 
\left(  u^\prime v^\prime  +\frac{1}{\Lambda}\frac{v^{p+1}}{p+1} 
+\frac{1}{\Lambda}\frac{u^{q+1}}{q+1} +\frac{\tilde{N}_+}{p+1} \frac{v u^\prime}{r}  +\frac{\tilde{N}_+}{q+1} \frac{u v^\prime}{r} \right) &\textrm{ in }\, \{u^{\prime\prime}>0\}\cap \{v^{\prime\prime}>0\} .\smallskip
\end{cases}
\end{align*}
Equivalently, in terms of the variables $X,Y,Z,W$, the energy $E(t)=E_{\sigma,A_i}(t,X,Y,Z,W)$ reads as
\begin{align*}
E(t)=
\begin{cases}
e^{t(N-2-\alpha-\beta)}  (XZ)^{\frac{\beta}{2}} (YW)^{\frac{\alpha}{2}} 
\left( {XY}+\frac{XZ}{\lambda (p+1)}+\frac{YW}{\lambda (q+1)}-\frac{NX}{p+1}-\frac{NY}{q+1}
\right) & \hspace{-0.1cm}\textrm{ in } R^+_{\lambda,Z}\cap R^+_{\lambda,W}, \medskip
\\
e^{t(\sigma-2-\alpha-\beta)}  (XZ)^{\frac{\beta}{2}} (YW)^{\frac{\alpha}{2}} 
\left( {XY}+\frac{XZ}{\lambda (p+1)}+\frac{YW}{\Lambda (q+1)}-\frac{N X}{p+1}-\frac{\tilde{N}_+ Y}{q+1}
\right) & \hspace{-0.1cm} \textrm{ in } R^+_{\lambda,Z}\cap R^-_{\lambda,W}, \medskip
\\
e^{t(\sigma-2-\alpha-\beta)}  (XZ)^{\frac{\beta}{2}} (YW)^{\frac{\alpha}{2}} 
\left( {XY}+\frac{XZ}{\Lambda (p+1)}+\frac{YW}{\lambda (q+1)}-\frac{\tilde{N}_+ X}{p+1}-\frac{N Y}{q+1}
\right)  & \hspace{-0.1cm}\textrm{ in } R^-_{\lambda,Z}\cap R^+_{\lambda,W}, \medskip
\\
e^{t(\tilde{N}_+-2-\alpha-\beta)}  (XZ)^{\frac{\beta}{2}} (YW)^{\frac{\alpha}{2}} 
\left( {XY}+\frac{XZ}{\Lambda (p+1)}+\frac{YW}{\Lambda (q+1)}-\frac{\tilde{N}_+ X}{p+1}-\frac{\tilde{N}_+ Y}{q+1}
\right)  & \hspace{-0.1cm}\textrm{ in } R^-_{\lambda,Z}\cap R^-_{\lambda,W}.
\end{cases}
\end{align*}

\smallskip

We observe that $E(t,O)=E(t,N_0)=E(t,A_0)=E(t,K_0)=E(t,L_0)\equiv  0$ for all $t$. Besides,
\begin{center}
$E(t,M_0)=-\frac{4}{pq-1}e^{t(\tilde{N}_+-2-\alpha-\beta)}<0$\;\; for all $t\in \real$,
\end{center}
thus a trajectory which converges to $M_0$ attains the energy value
\begin{align}\label{energia M0}
\textstyle E(M_0)<0 \;\textrm{ on } \tH_+ , \;\;E(M_0)=0 \;\textrm{ below } \tH_+ , \;\; E(M_0)=-\infty \;\textrm{ above } \tH_+ ,\;\;
\end{align}
by understanding the previous expressions as limits as $t\to +\infty$, for $\tH_+$ defined in \eqref{tH}.
Also, \vspace{-0.2cm}
\begin{center}
$E(t,P_0)=-(\tilde{N}_+-2)e^{t(\tilde{N}_+-2-\alpha-\beta)} \{\frac{\tilde{N}_+}{p+1} - \frac{-2+q(\tilde{N}_+-2)}{q+1} \}, $ 
\medskip

$E(t,Q_0)=-(\tilde{N}_+-2)e^{t(\tilde{N}_+-2-\alpha-\beta)} \{\frac{\tilde{N}_+}{q+1} - \frac{-2+p(\tilde{N}_+-2)}{p+1} \}. $
\end{center}
In particular, if a trajectory $\tau(t)$ converges either to $P_0$ or $Q_0$ as $t\to +\infty$, then
\begin{align}\label{energia P0,Q0}
E(P_0)=0, \;\; E(Q_0)=0 \;\;\textrm{ below } \tH_+.
\end{align}

Notice that the energy (defined in $\K$) is not continuous in the variables $X,Y,Z,W$ when crossing the concavity hyperplanes $\pi_{\lambda,Z}$ and $\pi_{\lambda, W}$.
However, the expressions in brackets are so, and the energy preserves sign when the respective solutions $u,v$ change concavity. 

Moreover, the derivative of $E$ with respect to $r=e^t$ is given by
\begin{align*}
E^\prime(r)=
\begin{cases}
\,{r^{N-1}} \,u^{\prime}v^{\prime}\,
\left\{ \frac{N}{p+1}+\frac{N}{q+1}-(N-2)
\right\} \; &\textrm{ in } R^+_{\lambda,Z}\cap R^+_{\lambda,W}, \medskip
\\
\,{r^{\tilde{N}-1}} \,u^{\prime}v^{\prime}
\,
\left\{ \frac{\tilde{N}_\pm}{p+1}+\frac{\tilde{N}_\pm}{q+1}-(\tilde{N}_\pm-2)
\right\} \; &\textrm{ in } R^-_{\lambda,Z}\cap R^-_{\lambda,W} .
\end{cases}
\end{align*}
Now, since regular positive solutions satisfy $u^\prime, v^\prime <0$, it follows that
\begin{align}
\dot{E}\le   0 \; \textrm{\;\; above or on $\tH_\pm$ \;\; in \,}& R^+_{\lambda,Z}\cap R^+_{\lambda,W} \textrm{\; and\;\; in \,} R^-_{\lambda,Z}\cap R^-_{\lambda,W}  \label{sign Eprime<0},\\
\dot{E}>0 \; \textrm{\;\; below $\sH$\;\;  in \,}&  R^+_{\lambda,Z}\cap R^+_{\lambda,W} \textrm{\; and\;\; in \,} R^-_{\lambda,Z}\cap R^-_{\lambda,W}  
\label{sign Eprime>0}.
\end{align}
Furthermore, in the mixed regions where the dependence of the parameter $\sigma$ appears, one finds
\begin{align}\label{eq Esigma prime}
\textstyle
E^\prime(r)=
\begin{cases}
\,{r^{\sigma-1}} \,u^{\prime}v^{\prime}\,
(2-N-\tilde{N}_+ +\sigma+ \frac{N}{p+1}+\frac{\tilde{N}_+}{q+1}
) +  r^{\sigma-1} \,\mathcal{E}_{\sigma,1} (r) \quad \textrm{ in } R^+_{\lambda,Z}\cap R^-_{\lambda,W},
\vspace{0.3cm}
\\
\textstyle
\,{r^{\sigma-1}} \,u^{\prime}v^{\prime}\,
(2-N-\tilde{N}_+ +\sigma+ \frac{\tilde{N}_+}{p+1}+\frac{N}{q+1}
) +  r^{\sigma-1} \,\mathcal{E}_{\sigma,2} (r) \quad \textrm{ in } R^-_{\lambda,Z}\cap R^+_{\lambda,W},
\end{cases}
\end{align}
where
\vspace{-0.3cm}
\begin{align*}
\textstyle
\mathcal{E}_{\sigma,1}(r)=
\{\frac{v^{p+1}}{\lambda (p+1)}+
\frac{N vu^\prime}{r(p+1)} \}\,
(\sigma - N) 
+\{\frac{u^{q+1}}{\Lambda(q+1)}
+\frac{\tilde{N}_+ v^\prime u}{r(q+1)} \}\,
(\sigma -\tilde{N}_+),
\end{align*}
\vspace{-0.5cm}
\begin{align*}
 \textstyle
\mathcal{E}_{\sigma,2}(r)=
\{\frac{v^{p+1}}{\Lambda (p+1)}
+\frac{\tilde{N}_+ vu^\prime}{r(p+1)}\}\,(\sigma - \tilde{N}_+) 
+\{\frac{u^{q+1}}{\lambda(q+1)}
 +\frac{N v^\prime u}{r(q+1)}\}\,(\sigma -N),
\end{align*}
or in terms of the dynamical system variables \eqref{X,Y,Z,W sem a},
\begin{align}\label{Eps sigma,1,2 X}
\textstyle
\mathcal{E}_{\sigma,1}(t)=
{e^{-t(2+\alpha+\beta)}  (XZ)^{\frac{\beta}{2}} (YW)^{\frac{\alpha}{2}} } \left\{\,
\frac{\sigma - N}{p+1} \,X(\frac{Z}{\lambda}-N)+ \frac{\sigma -\tilde{N}_+}{q+1}\,Y(\frac{W}{\Lambda}-\tilde{N}_+)\,
\right\},\nonumber \\
\textstyle
\mathcal{E}_{\sigma,2}(t)=
{e^{-t(2+\alpha+\beta)}  (XZ)^{\frac{\beta}{2}} (YW)^{\frac{\alpha}{2}} } \left\{\,
\frac{\sigma - \tilde{N}_+}{p+1} \,X(\frac{Z}{\Lambda}-\tilde{N}_+)+ \frac{\sigma -N}{q+1}\,Y(\frac{W}{\lambda}-{N})\,
\right\}.
\end{align}

\vspace{0.1cm}

On the other hand, for the operator $\M^-$ we consider the auxiliary hyperbolas
\begin{equation}\label{Hu e Hd-}
\overline{\sH}_- \, : \;\; \frac{1}{p+1}+\frac{1}{q+1}=\frac{N -2}{\tilde{N}_-},
\qquad
\underline{\sH}_- \, : \;\; \frac{1}{p+1}+\frac{1}{q+1}=\frac{\tilde{N}_--2}{N},
\end{equation}
which are contained in the regions $\mathcal{R}^-_u$ in \eqref{Ru-} and $\mathcal{R}^-_d$ in \eqref{Rd-}, respectively.

\vspace{0.1cm}

The energy for the system involving $\M^-$ is established by exchanging the roles of $\lambda$ and $\Lambda$ in the definition of the energy $E(r)$ for $\M^+$.

\begin{proof}[Proof of Theorem \ref{Th regular}]
	
Step 1) Existence of ground state solutions in $\overline{\mathcal{R}}_u^+$.

	Let us first consider the operator $\M^+$.
	We fix a pair $(p,q)$ in the region $\mathcal{R}_u^+$. Recall $\mathcal{R}_u^+$ is contained in the region lying above or on $\tH_+$.
	
	\vspace{0.1cm}
	
	We infer that $E$ is nonincreasing for all $t\in (-\infty,T]$. Indeed, we set $\sigma=N$, then 
	\begin{align}\label{energia resto decrescente}
	\textstyle
	\mathcal{E}_{N,1}(t)=
	{e^{-t(2+\alpha+\beta)}  (XZ)^{\frac{\beta}{2}} (YW)^{\frac{\alpha}{2}} } \left\{\,
	\frac{N-\tilde{N}_+}{q+1}\,Y(\frac{W}{\Lambda}-\tilde{N}_+)\,
	\right\} \le 0 ,\nonumber \\
	\textstyle
	\mathcal{E}_{N,2}(t)=
	{e^{-t(2+\alpha+\beta)}  (XZ)^{\frac{\beta}{2}} (YW)^{\frac{\alpha}{2}} } \left\{\,
	\frac{N - \tilde{N}_+}{p+1} \,X(\frac{Z}{\Lambda}-\tilde{N}_+) \,
	\right\}\le 0,
	\end{align}
since $\tilde{N}_+\leq N$ and Proposition \ref{prop flow} (iii), from which  $Z,\,W<\lambda N\leq \Lambda \tilde{N}_+$.
This fact and the inequalities coming from the region \eqref{Ru} applied to \eqref{eq Esigma prime} yield $E^\prime \le 0$ in the mixed regions. Since $\mathcal{R}_u^+$ in \eqref{Ru} is a region located above the hyperbolas \eqref{sH} and \eqref{tH}, with \eqref{sign Eprime>0} being true, then the desired monotonicity follows.

We assume by contradiction that there not exists any regular solution in $\rN$ of \eqref{LE} for this pair $(p,q)$. 
Then, by Theorem \ref{teo uniqueness trajectory}, there exists a regular solution $(u,v)$ of \eqref{LE}, \eqref{H Dirichlet} in the ball $B_R$, with corresponding regular trajectory $\Gamma(t)=(X(t),Y(t),Z(t),W(t))$ of \eqref{DS+} such that $X(T)=Y(T)=+\infty$ for $T=\mathrm{ln} (R)$. 	

Recall that there exists a one to one correspondence between solutions of \eqref{LE} and trajectories of \eqref{DS+}, in particular between their energies.
We know that the solution $(u,v)$ starts at $r=0$ with zero energy. Thus, as a nonincreasing function, the energy of $(u,v)$ remains nonpositive whenever it is defined, in particular at the limit point.

On the other hand, observe that \begin{center}
	$u^{\prime\prime}(R)=M(-\lambda r^{-1} (N-1)u^\prime)>0$, \; and\; $v^{\prime\prime}(R)=M(-\lambda r^{-1} (N-1)v^\prime)>0$.
\end{center} 
Then at $r=R$ we have 
$E(R)=r^{\tilde{N}_+} u^\prime (R) v^\prime (R)>0$. This yields a contradiction with $E(R)\le 0$.
So by Theorem \ref{teo uniqueness trajectory} there exists a solution of \eqref{LE} in $\rN$.
	
	\vspace{0.05cm}
	
	The proof for $\M^-$ in \eqref{Ru-} is analogous, by taking instead $\sigma=\tilde{N}_- = \max\{N,\tilde{N}_-\}$.

\medskip

Step 2) Nonexistence of ground states in $\mathcal{R}_d^+$.

We consider the operator $\M^+$.
Set $\sigma=\tilde{N}_+$.
With this choice in \eqref{eq Esigma prime} and \eqref{sign Eprime>0}, 
we infer that $E$ is a nondecreasing function for all $t\in \real$ when the pair $p,q$ lies in the region $\mathcal{R}_d^+$ in \eqref{Rd}. Indeed, by \eqref{Eps sigma,1,2 X},
\begin{align}\label{energia resto crescente}
\textstyle
\mathcal{E}_{\tilde{N}_+,1}(t)=
{e^{-t(2+\alpha+\beta)} \, (XZ)^{\frac{\beta}{2}} (YW)^{\frac{\alpha}{2}} } \left\{\,
\frac{\tilde{N}_+- N}{p+1} \,X(\frac{Z}{\lambda}-N)\,
\right\}\ge 0,\nonumber \\
\textstyle
\mathcal{E}_{\tilde{N}_+,2}(t)=
{e^{-t(2+\alpha+\beta)}\,  (XZ)^{\frac{\beta}{2}} (YW)^{\frac{\alpha}{2}} } \left\{\,
\frac{\tilde{N}_+ -N}{q+1}\,Y(\frac{W}{\lambda}-{N})\,
\right\}\ge 0,
\end{align}
since $\tilde{N}_+\leq N$, and $Z,W<\lambda N$ by Proposition \ref{prop flow} (iii).

Assume by contradiction that there exists a nontrivial positive regular solution in $\rN$. Then the correspondent regular trajectory $\Gamma$ starts at $-\infty$ with zero energy from $N_0$. 

Note that it holds the bound $E(r)\leq Cr^{N-2-\alpha-\beta}$ for large $r$ when $\sigma=\tilde{N}_+$, by Proposition \ref{AP bounds} and $\tilde{N}_+\leq N$. Then, since the region $\mathcal{R}_d^+$ in \eqref{Rd} is situated below the hyperbola $\sH$, one has $\lim_{r\rightarrow\infty}E(r)=0$.
But this contradicts the monotonicity of $E$ in the case of a nontrivial pair solution $u,v$ in $\rN$.
So, by Theorem \ref{teo uniqueness trajectory} there exists a solution of the Dirichlet problem in a ball.

The results for $\M^-$ in the region \eqref{Rd-} are analogous, by taking $\sigma=N=\min\{N,\tilde{N}_-\}$.
\end{proof}

\begin{lem}\label{lemma X,Y extremum}
Let $\Gamma=(X,Y,Z,W)$ be a regular trajectory of \eqref{DS+}. If 
\begin{align}\label{hip energia geq 0}
\textrm{$E\geq 0$\;\; in $R^-_{\lambda,Z}\cap R^-_{\lambda,W}$}
\end{align}
then $X$ and $Y$ may have at most one extremal point, which is a maximum. Moreover,
\begin{enumerate}[(i)]
\item if the trajectory $\Gamma$ is defined for all $t$, then it converges to a stationary point;

\item if instead $\Gamma$ blows up in finite time, then $\dot{X},\dot{Y}>0$ in its whole interval of definition. 
\end{enumerate}
An analogous result holds for the system \eqref{DS-} with respect to the operator $\M^-$.
\end{lem}

\begin{proof} Recall $\dot{X}>0$ in $R^+_{\lambda,Z}$ and $\dot{Y}>0$ in $R^+_{\lambda,W}$.
We observe that hypothesis \eqref{hip energia geq 0} yields \vspace{-0.1cm} 
	\begin{align*}
	\textstyle {XY}+\frac{XZ}{\Lambda (p+1)}+\frac{YW}{\Lambda (q+1)}-\frac{\tilde{N}_+ X}{p+1}-\frac{\tilde{N}_+ Y}{q+1}\geq 0 \;\;\textrm{ in }R^-_{\lambda,Z}\cap R^-_{\lambda,W},
	\end{align*}
	and so, by \eqref{DS+},
	\begin{equation}\label{eq Ydot lemX,Y max}
	\textstyle \dot{Y}=Y(  Y+2-\tilde{N}_+ + \frac{W}{\Lambda})\geq Y[  Y+2-(q+1)X]  +\frac{q+1}{p+1}\,X(\tilde{N}_+ -\frac{Z}{\Lambda}) \;\textrm{ in $R^-_{\lambda,Z}\cap R^-_{\lambda,W}$}.
	\end{equation}
Using \eqref{DS+} again, we get
	\begin{equation}\label{eq Xdot lemX,Y max}
	\textstyle
	\dot{X}=X(  X+2-\tilde{N}_+ + \frac{Z}{\Lambda})  , \quad \dot{Z}=Z(  \tilde{N}_+ -p Y-\frac{Z}{\Lambda}) \textrm{\quad in $R^-_{\lambda,Z}$}.
	\end{equation}
	By Lemma \ref{lemma u,v concave at 0}, $\Gamma$ starts from $N_0$ in the region $R^+_{\lambda,Z}\cap R^+_{\lambda,W}$. Suppose that $X$ has a maximum at $t_{0}$ followed by a minimum at $t_{1}.$ At these times we have $\dot{X}=0$, and so $\Gamma(t_0),\Gamma(t_1)\in R^-_{\lambda,Z}$. 
From \eqref{eq Xdot lemX,Y max} we have
	$\Lambda\partial_{tt} X=X\dot{Z}$\, thus $\dot{Z}(t_{0})\le 0 \le \dot{Z}(t_{1})$. 
There exists
	$t_{2}\in[  t_{0},t_{1}]  $ such that $\dot{Z}(t_{2})=0$, with $t_{2}$
	being a minimum. We obtain $\dot{X}(t_{2})\le 0$ by construction, and so $\Gamma(t_2)\in R^-_{\lambda,Z}$. By definition \eqref{pi 3,4,lambda} one has \vspace{-0.1cm}
	\begin{align}\label{(1) lemmaX,Y extremum}
	\textstyle \frac{Z}{\Lambda}(t_{2})=\tilde{N}_+-p Y(t_{2}),
	\end{align}
	then $0\le Z_{tt}(t_{2})=-p(Z\dot{Y})(t_{2})$, so $\dot{Y}(t_2)\le 0$. In particular, $\Gamma(t_2)\in R^-_{\lambda,W}$, besides
\begin{equation}\label{Y<tildeN-2}
Y(t)< \tilde{N}_+-2 \textrm{ for all } t\leq t_2
\end{equation} 
by Proposition \ref{prop flow} (ii).
We infer that
\begin{align}\label{X(t0)}
X(t_0)<p(\tilde{N}_\pm-2)-2.  
\end{align} 
Indeed, by putting $\dot{X}(t_0)=0$ and $\dot{Z}(t_0)\le 0$ in \eqref{eq Xdot lemX,Y max} one has $\tilde{N}_+-pY(t_0)\le \frac{Z(t_0)}{\Lambda}= \tilde{N}_+-2-X(t_0)$; then we get \eqref{X(t0)} by \eqref{Y<tildeN-2}. 
Further, \eqref{eq Ydot lemX,Y max}, \eqref{(1) lemmaX,Y extremum} yield
	\begin{align}\label{(2) lemmaX,Y extremum}
	\textstyle 0 \ge \dot{Y}(t_{2})\ge Y(t_{2})\, [\,  Y(t_{2})+2-\frac{q+1}{p+1}X(t_{2})\,]  .
	\end{align}
Again by $\dot{X}(t_{2})\le 0$ one gets $(X+\frac{Z}{\Lambda})(t_{2})\le \tilde{N}_+-2$. This and \eqref{(1) lemmaX,Y extremum}, \eqref{(2) lemmaX,Y extremum} imply
	\[
	\textstyle \tilde{N}_+-2-X(t_{2})\ge \frac{Z}{\Lambda}(t_{2})\ge \tilde{N}_+-p\left(\frac{q+1}{p+1}X(t_{2})-2\right)
	\] 
	from which
	\begin{align}\label{(3) lemmaX,Y extremum}
	\textstyle 2(p+1)\le \left( p\,\frac{q+1}{p+1}-1\right)X(t_{2} )=\frac
	{2(p+1)}{(\tilde{N}_+-2)p-2}X(t_{2}) \;\;\textrm{ i.e.\ }\; X(t_{2})\ge p(\tilde{N}_+ -2)-2.
	\end{align}
On the other hand $X(t_{2})\le X(t_{0})<p(\tilde{N}_+ -2)-2$ by \eqref{X(t0)}, which contradicts \eqref{(3) lemmaX,Y extremum}. Then $X$	has at most one extremum. 
Hence, either $X$ is an increasing function, or it has one extremum. In the latter case this extremum is a maximum, so $\dot{X}<0$ for some time on. 
By symmetry of $X$ and $Y$ with respect to their equations, also $Y$ has at most one extremum. 

Concerning $(i)$, independently of having an extremum or not, $X$ has a finite limit in $( 0,\tilde{N}_+-2]$ at $+\infty$, by Proposition \ref{AP bounds}.  
In this case, $Z$ also has at most one extremum, which is a minimum. Indeed, at the points where $\dot{Z}=0$ it happens that $-Z_{tt}$ has the sign of $\dot{Y}$. Thus $Z,W$ both have limits in the interval $[  0,\lambda N)  $. Therefore, $\Gamma$ converges to a stationary point.

Under $(ii)$ though, $\Gamma$ cannot have a maximum point since it is unbounded from above. 
\end{proof}

\begin{corol}\label{cor change concav 1x Rd}
	In $\mathcal{R}_d^\pm$ regular solutions of \eqref{P radial m} or \eqref{P radial m-} in a ball change concavity only once. 
\end{corol}

\begin{proof} Let us argue with the region $\mathcal{R}_d^+$ since for $\mathcal{R}_d^-$ it is analogous. By taking $\sigma=\tilde{N}_+$ in the expression of the energy, we obtain that the energy is nondecreasing in the region $\mathcal{R}_d^+$ whenever it is defined. Since a corresponding regular trajectory $\Gamma_p$ starts with zero energy at $N_0$ then hypothesis \eqref{hip energia geq 0} is verified. 
	
	Hence, hypothesis \eqref{hip energia geq 0} is verified and so Lemma \ref{lemma X,Y extremum} (ii) applies to ensure that
	\begin{equation}\label{dotX,Y>0}
	\textrm{$\dot{X}, \, \dot{Y}>0$ \; for all \,$t<T$.}
	\end{equation}
	
	Now we claim that $u,v$ change concavity only once.  
	Recall that they start concave in a neighborhood of $0$ by Lemma \ref{lemma u,v concave at 0}. If for instance $u$ changed concavity twice, then $\Gamma$ would cross the hyperplane $\pi_{\lambda,Z}$ at times $s_1,s_2$, first from $R^+_{\lambda,Z}$ to $R^-_{\lambda,Z}$ and then from $R^-_{\lambda,Z}$ to $R^+_{\lambda,Z}$. By Proposition~\ref{prop flow}~(i) we need to have at these times $Y(s_1)>1/p$ and $Y(s_2)<1/p$. But then $\dot{Y}(s_0)<0$ for some $s_0\in (s_1,s_2)$, which contradicts \eqref{dotX,Y>0}. The argument for $v$ is the same. 
\end{proof}

\begin{proof}[Proof of Theorem \ref{Th concavidade Lapl}]
We have already seen that a regular solution $(u,v)$ is such that $u,v$ change concavity at least once.	
When $\lambda=\Lambda$ we have $\tilde{N}_\pm=N$ and the operator is a multiple of the Laplacian. 

By taking $\sigma=N=\tilde{N}_\pm$, the energy for $\M^+=\M^-=\lambda \Delta$ is defined through a single continuous function, and so is monotonous with respect to the radius, see also \cite[proof of Theorem 1.6]{BV}.
Namely, the energy is always increasing in the subcritical case $\alpha + \beta <N-2$, while it is constant on the critical one $\alpha +\beta =N-2$. Anyway, the corresponding regular trajectory $\Gamma=(X,Y,Z,W)$ starts with zero energy  and remains nonnegative forward in time, in particular in the region $R^-_{\lambda,Z}\cap R^-_{\lambda,W}$. Thus, hypothesis \eqref{hip energia geq 0} is verified and Lemma \ref{lemma X,Y extremum} is applicable.
Under the subcritical case, the proof of Corollary \ref{cor change concav 1x Rd} with $\lambda=\Lambda$ already implies the desired conclusion.

In the critical case though, we are in the situation of Lemma \ref{lemma X,Y extremum} $(i)$, and $(p,q)\in \F$ by \cite[proof of Theorem 1.6]{BV}. 
Note that if we had $\dot{X}>0$ and $\dot{Y}>0$ for all $t\in \real$, then arguing as in the proof of of Corollary \ref{cor change concav 1x Rd} one already obtains the conclusion.

Thus, let us assume that at least one between $X$ or $Y$ has exactly one extremal point, which is a maximum. For instance, say it happens for $X$ at time $T$, then $\dot{X}>0$ for $t<T$, while $\dot{X}<0$ for all $t>T$. 

We claim that, since $X$ has a maximal point at $T$, then $v$ cannot change its concavity twice. 
Otherwise, if $v$ changed it at the points $t_1$ and $t_2$ consecutively, then at these points we would have $X(t_1)>1/q$, $X(t_2)<1/q$ by Proposition~\ref{prop flow} (i). Thus $X$ decreases somewhere in $(t_1,t_2)$ and so $T\in (t_1,t_2)$. Now, using that the hyperplane $\pi_{\lambda,W}$ is contained in the region where $\dot{X}>0$, we know that $\dot{X}(t_2)>0$, which contradicts the fact that $\dot{X}<0$ for all $t>T$.

Now it remains to be proved that $u$ does not change concavity twice. Recall that $Y$ also has at most one extremal point.
First, similarly to the preceding paragraph, one proves that if $Y$ has a maximal point at $T$ then $u$ cannot change its concavity twice.

To finish, suppose that $\dot{Y}>0$ for all $t\in \real$. If $\Gamma$ crossed $\pi_{\lambda,Z}$ at times $s_1$ and $s_2$ consecutively, we would have $Y(s_1)>1/p$ and $Y(s_2)<1/p$ by Proposition~\ref{prop flow} (i). Again, this means that $Y$ decreases at some point before reaching $s_2$. But by the mean value theorem, this yields the existence of a critical point $s_3\in (s_1,s_2)$ for $Y$, i.e.\ $\dot{Y}(s_3)=0$, which contradicts the fact that $Y$ has no extremal points.
The proof is then concluded.
\end{proof}

\begin{rmk} The situation of having one extremal point is only admissible if $\Gamma$ converges to either $P_0$ or $Q_0$. In other words, if $\Gamma$ converges to $A_0$ then both $X$ and $Y$ cannot have extremal points. Indeed, since $M_0=(\alpha,\beta,Z_0,W_0)\in \pi_{1,\lambda}\cap \pi_{3,\lambda}$ we have $X(T)<\alpha <\tilde{N}_\pm-2$ and $Y(T)<\beta <\tilde{N}_\pm-2$. So, in order to converge to $A_0$ at $+\infty$ both $X$ and $Y$ need to increase and reach $\tilde{N}_\pm -2$, instead of decreasing in some interval of time.
\end{rmk}

\begin{proof}[Proof of Theorem \ref{Th exterior}]
Step 1) Existence of fast decaying exterior domain solutions in $\mathcal{R}_u^\pm$.

We assume by contradiction the conclusion does not hold.
Then, by Theorem \ref{Th1.1 lemma A0} there exists a fast decaying trajectory $\Upsilon$, with corresponding solution $(u,v)$ of \eqref{LE} defined in the whole $\rN$. Thus, the a priori bounds in Proposition \ref{AP bounds} imply that $\Upsilon$ stays in the box $\mathcal{B}_+$, in particular satisfying $Z,W<\lambda N$ for all $t\in \real$. As a consequence, by \eqref{energia resto decrescente} the energy of $\Upsilon$ is a nonincreasing function.

Observe that $\Upsilon$ arrives at $+\infty$ with zero energy at the point $A_0$. Set $\sigma=N$. Then the energy of $\Upsilon$ remains nonnegative  whenever it is defined.
Once again we use Proposition \ref{AP bounds} to infer
\begin{center}
	$|E(\Upsilon,r) | \le C\, r^{N-2-\alpha-\beta} \to 0$\;\; as $r \to 0$,
\end{center}
since $\alpha +\beta <N-2$ represents the region above $\sH$, which in turn contains $\mathcal{R}_u^+$.
Thus $E(0)=0$, which is impossible for a nontrivial trajectory. 

\medskip

Step 2) Nonexistence in $\overline{\mathcal{R}}_D^\pm$.

Assume on the contrary that there exists a nontrivial exterior domain solution of \eqref{LE} in $\rN\setminus B_R$, for some $R>0$ with $\partial_\nu u =\partial_\nu v=0$ and $u,v=\kappa$ on $\partial B_R$ for some $\kappa >0$. 
Its corresponding trajectory $\Upsilon$ blows up at finite time $T=\mathrm{ln}(R)\in \real$ such that $\lim_{t\to T^+}Z(t)=\lim_{t\to T^+}W(t)=+\infty$. 
	
We fix the operator $\M^+$. 
For $\frak A , \frak B$ we define a new energy function as follows
\begin{align}\label{energia 2}
\mathbb E(r)=
\,r^{\tilde N_+}
\left(  u^\prime v^\prime  +\frac{1}{\Lambda}\frac{v^{p+1}}{p+1} 
+\frac{1}{\Lambda}\frac{u^{q+1}}{q+1}
+\frak A \frac{v u^\prime}{r}  + \frak B \frac{u v^\prime}{r} \right) .
\end{align}
Then we rewrite the problem \eqref{P radial m} as
\begin{center}
$u^{\prime\prime}+\frac{\hat N -1}{r}u^\prime = -\frac{v^p}{\hat \sigma}$, \qquad $v^{\prime\prime}+\frac{\bar N -1}{r}v^\prime = -\frac{u^q}{\bar \sigma}$,
\end{center}
where $(\hat N, \hat \sigma)$ and $(\bar N, \bar \sigma)$ are either $(N, \lambda)$ or $(\tilde N_+, \Lambda)$.
So we deduce
\begin{align*}
\textstyle
\mathbb E^\prime(r) &\ge
\,{r^{\tilde N_+-1}} \,u^{\prime}v^{\prime}\,
(2-\hat N -\bar N +\tilde N_+ +\frak A +\frak B) 
+ r^{\tilde N_+-2} \frak A (\tilde N_+ -\hat N) vu^\prime
+ r^{\tilde N_+-2} \frak B (\tilde N_+ -\bar N) uv^\prime 
\\ & \textstyle + v^{p+1} r^{\tilde N_+-1}\, \{\frac{1}{\Lambda}\frac{\tilde N_+}{p+1}- \frac{ \frak A}{\hat \sigma} \}
+ u^{q+1} r^{\tilde N_+-1}\, \{\frac{1}{\Lambda} \frac{\tilde N_+}{q+1}-\frac {\frak B}{\bar \sigma}\}.
\end{align*}

Let $\delta\ge 0$ be such that $\frac{1}{p+1}+\frac{1}{q+1}=  \frac{\Lambda}{\lambda} \frac{2N-\tilde N_+-2}{\tilde N_+} +\delta$. Then we define $\frak A :=\frac{\lambda}{\Lambda}  \frac{\tilde N_+}{p+1}-\frac{\delta}{2}$ and  $\frak B :=\frac{\lambda}{\Lambda}  \frac{\tilde N_+}{q+1}-\frac{\delta}{2}$. Hence $\frak A +\frak B = 2 N-\tilde N_+ -2$. 
These and $u^\prime, v^\prime <0$ yield
$\mathbb E^\prime (r) \ge 0$ for all $r\ge R$.\smallskip

Note that $E(\Upsilon,T)>0$, since $u^\prime (R)=v^\prime (R)=0$. Thus $E(t)>0$ for all $t\ge T$. 
	
Firstly, if $\Upsilon$ were a fast decaying trajectory, then it would end at $+\infty$ on either $A_0$, $P_0$ or $Q_0$. But all of these points have zero energy, see \eqref{energia P0,Q0}, which is impossible due to the monotonicity of $E$.
	
Now assume by contradiction $\Upsilon$ is a slow decaying trajectory i.e.\ $\omega(\Upsilon)=M_0$.
Since the region $\mathcal{R}_d^+$ is contained in the region below $\tH_+$, then $\Upsilon$ would end with zero energy by \eqref{energia M0}. But this contradicts the sign of the energy at the blow-up point $T$.
Finally, Lemma \ref{lemma X,Y extremum} concludes that no exterior domain oscillating solutions are admissible.

For $\M^-$ it is analogous, by considering the energy
\begin{align*}
\mathbb E(r)=
\,r^{N}
\left(  u^\prime v^\prime  +\frac{1}{\Lambda}\frac{v^{p+1}}{p+1} 
+\frac{1}{\Lambda}\frac{u^{q+1}}{q+1}
+\frak A \frac{v u^\prime}{r}  + \frak B \frac{u v^\prime}{r} \right) ,
\end{align*}
with $\frak A :=\frac{\lambda}{\Lambda}  \frac{N}{p+1}-\frac{\delta}{2}$ and  $\frak B :=\frac{\lambda}{\Lambda}  \frac{N}{q+1}-\frac{\delta}{2}$, whenever $\frac{1}{p+1}+\frac{1}{q+1}=  \frac{\Lambda}{\lambda} \frac{2\tilde N_--N-2}{\tilde N_-} +\delta$, for some $\delta \ge 0$.

\medskip

If $\lambda=\Lambda$, 
on the hyperbola $\sH$ in \eqref{sH}, the fast decaying trajectory corresponds to a solution defined in the whole $\rN$. Moreover, the energy of any trajectory of \eqref{DS+} is identically zero. This reduces the dimension of the system, where for instance the variable $W$ disappears.
It follows by \cite{BV, VanderUni} the analysis associated to the $3$-dimensional dynamical system in this case.
\end{proof}

\begin{proof}[Proof of Theorem \ref{Th singular}] The existence follows from item 1.1 in the proof of Theorem \ref{Th exterior}. Indeed, the nonexistence of fast decaying exterior domain solutions and the alternative in Theorem~\ref{Th1.1 lemma A0} leads to a fast decaying solution of \eqref{LE} in $\rN$. Similarly, the nonexistence part follows from Step 2 in the proof of Theorem \ref{Th exterior} Theorem~\ref{Th1.1 lemma A0}.
\end{proof}

\begin{proof}[Proof of Theorem \ref{Th exterior Dir}] Let $\lambda=\Lambda$ and $R>0$.
Assume by contradiction there exists a nontrivial exterior domain solution of \eqref{LE} in $\rN\setminus B_R$, with $u,v=0$ on $\partial B_R$. 

For $\frak A , \frak B$ we consider the energy \eqref{energia 2}. Now, since $N=\tilde N_+$, by taking $\delta\ge 0$ so that $\frac{N}{p+1}+\frac{N}{q+1}=N-2+\delta$, with  $\frak A = \frac{N}{p+1}-\frac{\delta}{2}$ and  $\frak B = \frac{N}{q+1}-\frac{\delta}{2}$, then we get 
$\mathbb E^\prime (r)=\frac{\delta r^{N-1}}{2\Lambda}(v^{p+1}+u^{q+1}) \ge 0$ for all $r\ge R$.
	
\medskip
	
Since  $\mathbb E$ is increasing and $\mathbb E(\Upsilon, R)>0$, we have $\mathbb E>0$ for all $r\ge R$. Now, as in Step 1 of the proof of Theorem \ref{Th exterior}, $\Upsilon$ cannot approach neither $A_0$, $P_0$, $Q_0$ nor $M_0$ at infinity because they have zero energy; while $\Upsilon$ is not oscillating at infinity by Lemma \ref{lemma X,Y extremum}. Hence, the existence of $\Upsilon$ is denied.
\end{proof}

\textbf{{Acknowledgments.}} L.\ Maia was supported by FAPDF, CAPES, and CNPq grant 309866/2020-0. G.\ Nornberg was supported by FAPESP grants 2018/04000-9 and 2019/03101-9, São Paulo Research Foundation. F.\ Pacella was supported by INDAM-GNAMPA.

\bibliography{bibtex}
\bibliographystyle{abbrv}

\end{document}